\RequirePackage{fix-cm}
\documentclass[smallextended]{svjour3}       
\smartqed  
\usepackage{graphicx}
\usepackage[hyphens]{url}
\usepackage{hyperref}
\usepackage{amsfonts}
\usepackage{amssymb}
\usepackage{latexsym}
\usepackage{graphicx}
\usepackage{mathtools}
\usepackage{color}
\usepackage{xcolor}
\usepackage{fullpage}

\usepackage{etoolbox}

\usepackage{booktabs}
\usepackage{mathscinet}

\usepackage{tikz}
\usepackage{pgfplots}
\usepackage{pgfplotstable}
\usetikzlibrary{external}

\usepackage{todonotes}

\newcommand*\circled[1]{\tikz[baseline=(char.base)]{
    \node[shape=circle,draw,inner sep=2pt] (char) {#1};}}

\def\F{\mathbb{F}}
\def\R{\mathbb{R}}
\def\C{\mathbb{C}}

\DeclareMathOperator{\trace}{tr}

\DeclareMathOperator{\vvec}{vec}

\DeclareMathOperator{\grad}{grad}
\DeclareMathOperator{\sk}{skew}

\DeclarePairedDelimiter{\norm}{\lVert}{\rVert}
\DeclarePairedDelimiter{\abs}{\lvert}{\rvert}

\spnewtheorem{theorem}{Theorem}[section]{\bfseries}{\itshape}

\spnewtheorem{lemma}[theorem]{Lemma}{\bfseries}{\itshape}
\spnewtheorem{definition}[theorem]{Definition}{\bfseries}{\itshape}
\spnewtheorem{remark}[theorem]{Remark}{\bfseries}{\upshape}
\spnewtheorem{example}[theorem]{Example}{\bfseries}{\upshape}
\spnewtheorem{corollary}[theorem]{Corollary}{\bfseries}{\itshape}
\spnewtheorem{proposition}[theorem]{Proposition}{\bfseries}{\itshape}
\spnewtheorem{algorithm}[theorem]{Algorithm}{\bfseries}{\upshape}

\numberwithin{theorem}{section}

\begin{document}

\title{Nearest $\Omega$-stable matrix via Riemannian optimization
}


\author{Vanni Noferini \and Federico Poloni}


\institute{V. Noferini \at
              Aalto University, Department of Mathematics and Systems Analysis, P.O. Box 11100, FI-00076 Aalto, Finland \\
              \email{vanni.noferini@aalto.fi}           
           \and
           F. Poloni \at
             Universit\`{a} di Pisa, Dipartimento di Informatica, Largo Bruno Pontecorvo 3, 56127 Pisa, Italy\\
             \email{federico.poloni@unipi.it}
}

\date{Received: date / Accepted: date}

\maketitle

\begin{abstract}
We study the problem of finding the nearest $\Omega$-stable matrix to a certain matrix $A$, i.e., the nearest matrix with all its eigenvalues in a prescribed closed set $\Omega$. Distances are measured in the Frobenius norm. An important special case is finding the nearest Hurwitz or Schur stable matrix, which has applications in systems theory. We describe a reformulation of the task as an optimization problem on the Riemannian manifold of orthogonal (or unitary) matrices. The problem can then be solved using standard methods from the theory of Riemannian optimization. The resulting algorithm is remarkably fast on small-scale and medium-scale matrices, and returns directly a Schur factorization of the minimizer, sidestepping the numerical difficulties associated with eigenvalues with high multiplicity.
\end{abstract}

\section{Introduction}
Let $\Omega$ be a non-empty closed subset of $\C$, and define
\begin{equation} \label{eq:feasible}   S(\Omega, n, \F) :=   \{  X \in \F^{n \times n}  : \Lambda(X) \subseteq \Omega \}  \subseteq \F^{n \times n}, \end{equation}
the set of $n \times n$ matrices (with entries in either $\F=\C$ or $\F=\R$) whose  eigenvalues all belong to $\Omega$. (Here $\Lambda(X)$ denotes the spectrum of the square matrix $X$.)

Given $A \in \F^{n \times n}$, we consider the problem of finding a matrix $B \in S(\Omega, n, \F)$ nearest to $A$, as well as the distance from $A$ to $B$. This is known as the \emph{nearest $\Omega$-stable matrix problem}. More formally, the problem is to find
\begin{equation} \label{minproblem}
	B = \arg \min_{X \in S(\Omega, n, \F)} \norm{A-X}_F^2
\end{equation}
together with the value of the minimum. The norm considered here is the Frobenius norm $\norm{M}_F := (\sum_{i,j=1}^n \abs{M_{ij}}^2)^{1/2}=\sqrt{\trace(M^*M)}$. 

Important examples of nearest $\Omega$-stable matrix problems include:
\begin{itemize}
\item the nearest Hurwitz stable matrix problem \footnote{In the control theory literature, a matrix satisfying our definition of Hurwitz stability is often referred to as Hurwitz \emph{semistable}, whereas the adjective stable is reserved for matrices in the interior of $S(\Omega_H, n, \F)$, i.e., matrices whose eigenvalues have strictly negative real part. However, open sets cannot possibly contain a nearest matrix to a given matrix $A$. Thus, given the context of this paper, we prefer for simplicity to simply define stable matrices as the matrices whose eigenvalues have nonpositive real part. Similar comments apply to other choices of $\Omega$.} ($\Omega = \Omega_H :=  \{  z \in \C : \Re z \leq 0   \}$), which arises in control theory \cite[Section 7.6]{Datta}. Here, $\Re z$ denotes the real part of $z\in\mathbb{C}$. Hurwitz stability is related to asymptotical stability of the dynamical system $\dot{x} = Ax$;
in some cases, numerical or modelling errors produce an unstable system in lieu of a stable one, and it is desirable to find a way to `correct' it to a stable one without modifying its entries too much. 
\item the nearest Schur stable matrix problem ($\Omega = \Omega_S := \{  z \in \C: |z| \leq 1   \}$), which is the direct analogue of the Hurwitz stable matrix problem that arises when, instead of a continuous-time, one has a discrete-time dynamical system $x_{k+1} = Ax_k$ \cite{Datta}.
\item the problem of finding nearest matrix with all real eigenvalues ($\Omega = \R$), which was posed (possibly more as a mathematical curiosity than for an engineering-driven application) as an open question in some internet mathematical communities~\cite{F17,Gary}.
\end{itemize}
Moreover, in the literature interest has been given also to more exotic choices. For instance, in \cite{CGS19} this problem is considered for the case were $\Omega$ is a region generated by the intersection of lines and circles.

Another common related (but different!) problem is that of finding the nearest $\Omega$-\emph{unstable} matrix, i.e., given a matrix $A$ with all its eigenvalues in the interior of a closed set $\Omega$, finding the nearest matrix with \emph{at least} one eigenvalue outside the interior of $\Omega$, together with its distance from $A$. For $\Omega=\Omega_H$ (resp. $\Omega = \Omega_S$), this minimal distance (known as \emph{stability radius}) is related to the \emph{pseudospectral abscissa} (resp. \emph{pseudospectral radius}), respectively, and has been studied extensively; see for instance~\cite{BenM,BurLO,B88,GM15,GugO,HeW,HinP,KosMS,KreB,Men,MenO}. Other variants and extensions have been studied recently as well~\cite{GilKS19,GilKS20,GMS18,GugP18,NesP20}.

Nearest $\Omega$-stable matrix problems are notoriously hard. Unlike various other matrix nearness problems~\cite{Higham89matrixnearness}, and to our knowledge, there is no closed-form solution based on an eigenvalue or Schur decomposition. At least for the choices of $\Omega$ illustrated above, the feasible region $S(\Omega, n, \F)$ is nonconvex. As a result, currently existing algorithms cannot guarantee that a global minimizer is found, and are often caught in local minima. In the most popular case of the nearest Hurwitz stable matrix problem, many approaches have been proposed, including the following.
\begin{itemize}
	\item Orbandexivry, Nesterov and Van Dooren~\cite{ONV13} use a method based on successive projections on convex approximations of the feasible region $S(\Omega, n, \F)$.
	\item Curtis, Mitchell and Overton~\cite{CurMO}, improving on a previous method by Lewis and Overton~\cite{LewO}, use a BFGS method for non-smooth problems, applying it directly to the largest real part among all eigenvalues.
	\item Choudhary, Gillis and Sharma~\cite{CGS19,GS17} use a reformulation using dissipative Hamiltonian systems, paired with various optimization methods, including semidefinite programming.
\end{itemize}

In this paper, we consider the problem for a completely general $\Omega$ and we describe a novel approach: we parametrize $X$ with its complex Schur factorization $X=UTU^*$, and observe that, if $U$ is fixed, then there is an easy solution to the simplified problem in the variable $T$ only. As a remarkable consequence, we demonstrate that finding an $\Omega$-stable matrix nearest to $A$ is equivalent to minimizing a certain function (depending both on $\Omega$ and on $A$) over the matrix Riemannian manifold of unitary matrices $U(n)$. A version of the method employing only real arithmetic (and minimizing over the manifold of orthogonal matrices $O(n)$) can be developed with some additional considerations.

This reformulation of the problem is still difficult. Indeed, local minima of the original problem are also mapped to local minima of the equivalent Riemannian optimization problem. As a result, we cannot guarantee a global optimum either. However, there are several advantages coming from the new point of view:

\begin{enumerate}
\item The approach is completely general, and unlike some previously known algorithms it works for any closed set $\Omega$;
\item It can exploit the robust existing machinery of general algorithms for optimization over Riemannian manifolds~\cite{AMS08,BMAS14};
\item In the most popular practical case of Hurwitz stability, numerical experiments show that our algorithm performs significantly better than state-of-the-art existing algorithms, both in terms of speed and in terms of accuracy, measured as the distance of the (possibly not global) minimum found from the matrix $A$;
\item Our algorithm first finds an optimal unitary/orthogonal matrix $Q$, then performs (implicitly) a unitary/orthogonal similarity $A \mapsto Q^*A Q$, and finally projects in a simple way onto $S(\Omega, n, \F)$: the simplicity of this procedure ensures backward stability of the computation of $B$.
\item The algorithm produces as an output directly a Schur decomposition $B = QTQ^*$ (or a close analogue) of the minimizer. This decomposition is useful in applications, and is a `certificate' of stability. Chances are that the problem of computing the Schur decomposition \emph{a posteriori} given $B$ is a highly ill-conditioned one, since in many cases the minimizer $B$ has multiple eigenvalues with high multiplicity. Hence, it is convenient to have this decomposition directly available.
\end{enumerate}

The paper is structured as follows. After some preliminaries concerning the distance from closed subsets of $\R^N$ (Section \ref{sec:distance}), in Section \ref{sec:complexsetup} we describe the theory underlying our method for complex matrices, while in Section \ref{sec:realsetup} we set up the slightly more involved theoretical background for real matrices. Section~\ref{sec:real2by2} deals with the $2\times 2$ case of the problem: we compute a closed-form solution for the cases $\Omega_H$ and $\Omega_S$. Section \ref{sec:gradient} describes the details of how to compute the gradient of the objective function in our reformulation of the problem, which is key for devising a practical algorithm. In Section \ref{sec:numexp}, we describe the results of some numerical experiments, comparing (with rather promising outcome) with existing methods. We finally draw some conclusions in Section \ref{sec:conclusions}.

The code that we used for our numerical experiments is made publicly available from the repository \url{https://github.com/fph/nearest-omega-stable}.

\section{Squared distance from closed sets}\label{sec:distance}

Let $\Omega \subseteq \mathbb{C}$ be a non-empty closed set, and define
\begin{align*}
p_\Omega &: \mathbb{C} \to \Omega, &  p_\Omega(z)& = \arg \min_{x\in \Omega} \abs{z-x}^2,\\
d^2_\Omega &: \mathbb{C} \to \mathbb{R}, & d^2_\Omega(z) &= \min_{x\in \Omega} \abs{z-x}^2.
\end{align*}
Note that $p_\Omega$ is not always uniquely defined: for instance, take $\Omega = \{0,2\} \subset \C$ and let $z$ be any point with $\Re z = 1$. In many practical cases the minimum is always achieved in a unique point (for instance, when $\Omega$ is convex), but in general there may be a non-empty set of points $z$ for which the minimizer is non-unique, known as \emph{medial axis} of $\Omega$. 

The same definitions can be formulated for $\mathbb{R}^N$, and indeed the definitions on $\mathbb{C}$ are just a special case of the following ones, thanks to the isomorphism $\mathbb{C} \simeq \mathbb{R}^2$. Let now $\Omega \subseteq \mathbb{R}^N$ be a non-empty closed set, and define
\begin{subequations} \label{distRN}
\begin{align}
p_\Omega &: \mathbb{R}^N \to \Omega, &  p_\Omega(z)& = \arg \min_{x\in \Omega} \norm{z-x}^2,\\
d^2_\Omega &: \mathbb{R}^N \to \mathbb{R}, & d^2_\Omega(z) &= \min_{x\in \Omega} \norm{z-x}^2.
\end{align}
\end{subequations}

The following result proves differentiability of the squared distance function $d_\Omega^2$: it is stated in \cite[Proposition~4.1]{chazal} for a compact $\Omega \subseteq \mathbb{R}^N$, but it is not difficult to see that it holds also for unbounded closed sets, since it is a local property.
\begin{theorem} \label{thm:distance-function}
The squared distance function $d_\Omega^2(z)$ is continuous on $\mathbb{R}^N$ (or $\mathbb{C}$), and almost everywhere differentiable. Indeed, it is differentiable on the complement of the medial axis of $\Omega$. Its gradient is $\nabla_z d_\Omega^2(z) = 2(z-p_\Omega(z))$.
\end{theorem}
In this paper, we will mostly be concerned in cases in which the set $\Omega$ is convex, so the medial axis is empty, but most of the framework still works even in more general cases. When the medial axis is not empty, we henceforth tacitly assume that, for all $z$ belonging to the medial axis, $p_\Omega(z)$ has been fixed by picking one of the possible minimizers (if necessary, via the axiom of choice).

\section{Reformulating the problem: the complex case} \label{sec:complexsetup}

We start with a lemma that shows how to solve~\eqref{minproblem} with the additional restriction that $X$ is upper triangular.

\begin{lemma} \label{lem:aux}
Let $\hat{A}\in\mathbb{C}^{n\times n}$ be given. Then, a solution of
\begin{equation} \label{auxmin}
	\mathcal{T}(\hat{A}) = \arg \min_{\substack{T \in S(\Omega,n,\C) \\  \text{$T$ upper triangular}}} \norm{\hat{A} - T}_F^2	
\end{equation}
is given by 
\begin{equation} \label{T}
\mathcal{T}(\hat{A})_{ij} = \left\{\begin{array}{lll}
    \hat{A}_{ij} & i<j & \text{(upper triangular part)},\\
    p_\Omega(\hat{A}_{ii}) & i=j & \text{(diagonal part)},\\
    0 & i > j & \text{(lower triangular part)}.\\
\end{array}	\right.
\end{equation}
\end{lemma}
\begin{proof}
We have
\[
\norm{\hat{A} - T}_F^2 = \underbrace{\sum_{i>j} \abs{\hat{A}_{ij}}^2}_{\circled{1}} + \underbrace{\sum_{i} \abs{\hat{A}_{ii}- T_{ii}}^2}_{\circled{2}} + \underbrace{\sum_{i<j} \abs{\hat{A}_{ij}- T_{ij}}^2}_{\circled{3}}.
\]
Clearly, \circled{1} is constant, the minimum of~\circled{2} under the constraint that $T_{ii} \in \Omega$ is achieved when $T_{ii} = p_{\Omega}(\hat{A}_{ii})$,
and the minimum of~\circled{3} is achieved when $T_{ij}=\hat{A}_{ij}$.
\end{proof}

In particular, it holds that 
\[
\min_{\substack{T \in S(\Omega,n,\C) \\  \text{$T$ upper triangular}}} \norm{\hat{A} - T}_F^2 = \norm{\hat{A} - \mathcal{T}(\hat{A})}_F^2 = \norm{\mathcal{L}(\hat{A})}_F^2,
\]
with $\mathcal{L}(\hat{A}) = \hat{A} - \mathcal{T}(\hat{A})$. One can see that the matrix $\mathcal{L}(\hat{A})$ is lower triangular, and has entries
\begin{equation} \label{L}
\mathcal{L}(\hat{A})_{ij} =
 \begin{cases}
    0 & i<j,\\
    \hat{A}_{ii} - p_\Omega(\hat{A}_{ii}) & i=j,\\
    \hat{A}_{ij} & i > j.
\end{cases} 
\end{equation}

Note that the matrices $\mathcal{T}(\hat{A})$ and $\mathcal{L}(\hat{A})$ are uniquely defined if and only if $p_\Omega(\hat{A}_{ii})$ is unique for each $i$. However, the quantity $\norm{\mathcal{L}(\hat{A})}_F^2$, which is the optimum of~\eqref{auxmin}, is always uniquely determined, since $\abs{\hat{A}_{ii} - p_\Omega(\hat{A}_{ii})}^2 = d_\Omega^2(\hat{A}_{ii})$ is uniquely determined.

Thanks to Lemma~\ref{lem:aux}, we can convert the nearest $\Omega$-stable matrix problem~\eqref{minproblem} into an optimization problem over the manifold of unitary matrices
\begin{equation} \label{minriemann2}
	\min_{U \in U(n)} \norm{\mathcal{L}(U^*AU)}^2_F.
\end{equation}
Indeed, Theorem \ref{thm:equivalence} below shows the equivalence of \eqref{minproblem} and \eqref{minriemann2}.

\begin{theorem} \phantom{A}\label{thm:equivalence}
\begin{enumerate}
	\item The optimization problems~\eqref{minproblem} and~\eqref{minriemann2} have the same minimum value.
	\item If $U$ is a local (resp. global) minimizer for~\eqref{minriemann2}, then $B = UTU^*$, where $T = \mathcal{T}(U^*AU)$, is a local (resp. global) minimizer for~\eqref{minproblem}.
	\item If $B$ is a local (resp. global) minimizer for~\eqref{minproblem}, and $B=UTU^*$ is a Schur decomposition, then $U$ is a local (resp. global) minimizer for~\eqref{minriemann2} and $T = \mathcal{T}(U^*AU)$.
\end{enumerate}
\end{theorem}
\begin{proof}
Taking a Schur form of the optimization matrix variable $X=UTU^*$, we have
\begin{align*}
\min_{X \in S(\Omega,n,\C)} \norm{A-X}_F^2 &= \min_{U \in U(n)} \min_{\substack{\text{$T$ upper triangular}\\ \Lambda(T) \subseteq \Omega }} \norm{A-UTU^*}_F^2 \\
& = \min_{U \in U(n)} \min_{\substack{\text{$T$ upper triangular}\\ \Lambda(T) \subseteq \Omega }} \norm{U^*AU-T}_F^2\\
& = \min_{U \in U(n)} \norm{\mathcal{L}(U^*AU)}_F^2,
\end{align*}
where the last step holds because of Lemma~\ref{lem:aux}. All the statements follow from this equivalence.
\end{proof}

In addition, we can restrict the optimization problem \eqref{minriemann2} to the special unitary group $SU(n)$. Indeed, given $U \in U(n)$, defining $D=\mathrm{diag}(1,\dots,1,\det(U))$, we have $UD^* \in SU(n)$ and
\[  \norm{\mathcal{L}(U^*A U)}_F^2 = \norm{\mathcal{L}(DU^* A U D^*)}_F^2.   \]

\section{Reformulating the problem: the real case} \label{sec:realsetup}

The version of our result for $\F=\R$ is somewhat more involved. We can identify two separate cases, one being a faithful analogue of the complex case and the other involving some further analysis.

\subsection{$\Omega \subseteq \R$}

The simpler version of the real case is when $\Omega$ is a subset of the reals. This allows one to solve, for instance, the problem of the nearest matrix with real eigenvalues, $\Omega=\mathbb{R}$. 

In this case, matrices $X \in S(\Omega,n,\mathbb{R})$ have a real Schur form in which the factor $T$ is upper triangular (as opposed to quasi-triangular with possible $2\times 2$ blocks). In this case, \emph{mutatis mutandis}, all the arguments in Section~\ref{sec:complexsetup} still hold. Specifically, it suffices to replace $\C$ with $\R$, $U(n)$ with $O(n)$, $SU(n)$ with $SO(n)$ (observe in passing that $SO(n)$ is connected while $O(n)$ is not) and there is nothing more to say.

\subsection{$\Omega \not\subseteq \R$}

In the more general case, matrices $X \in S(\Omega,n,\mathbb{R})$ may have complex eigenvalues, so their real Schur form will have a quasi-triangular $T$. Since the zero pattern of $T$ is not uniquely determined, we need to modify slightly the approach of Section~\ref{sec:complexsetup} to be able to define a suitable objective function $f(Q)$. We solve this issue by imposing a specific zero pattern for $T$.
\begin{definition}
A real matrix $T \in \R^{n \times n}$ is said to be in \emph{modified real Schur form} if it is block upper triangular with all $2 \times 2$ blocks on the diagonal, except (if and only if $n$ is odd) for a unique $1 \times 1$ block at the right bottom.

A modified real Schur decomposition of a matrix $M \in \R^{n \times n}$ is $M=Q T Q^\top $ where $Q \in \R^{n \times n}$ is orthogonal and $T$ is in modified real Schur form.
\end{definition}
Since one can reorder blocks in the real Schur decomposition~\cite[Theorem 2.3.4]{HornJohnson}, in particular one can obtain a real Schur decomposition in which all non-trivial $2\times 2$ diagonal blocks are on top. Thus, the existence of modified real Schur decompositions is an immediate corollary of the existence of real Schur decompositions. Further modified Schur decompositions, having more non-trivial $2 \times 2$ blocks than the special ones described above, can be obtained by applying Givens rotations to appropriate $2\times 2$ triangular blocks of classical real Schur forms.

\begin{example}
The matrix
\[  T  = \begin{bmatrix}
1 & 2 & 3 & 4 & 5\\
6 & 7 & 8 & 9 & 0\\
0 & 0 & 1 & 3 & 7\\
0 & 0 & -4 & 2 & 6\\
0 & 0 & 0 & 0 & 5
\end{bmatrix}   \]
is in modified real Schur form. Note that in this example the top $2 \times 2$ block on the diagonal is associated with two \emph{real} eigenvalues, a case which would not be allowed in the classical real Schur form.
\end{example}
More formally, we partition $\{1,2,\dots,n\}$ into
\begin{equation} \label{partition}
\mathfrak{I} = \{\{1,2\}, \{3,4\}, \{5,6\}, \dots, F\}, \quad 
F = \begin{cases}
\{n-1,n\} & \text{$n$ even,} \\
\{n\} & \text{$n$ odd}.
\end{cases}    
\end{equation}
With this notation, $T$ is in modified real Schur form if $T_{IJ}=0$ for any two elements $I,J \in \mathfrak{I}$ of the partition such that $I>J$ (in the natural order).

Furthermore, to solve the real analogue of problem~\eqref{auxmin}, we will use a projection $p_{S(\Omega, 2, \mathbb{R})}$, i.e., a function that maps any matrix $A$ to a minimizer of the distance from $A$ among all $2 \times 2$ real matrices with eigenvalues in $\Omega$. Note that this is just a special case of ~\eqref{distRN}, with the (closed) set $S(\Omega,n,\R)$ in place of $\Omega$, since $\mathbb{R}^{2\times 2} \simeq \mathbb{R}^4$.

We can now state Lemma~\ref{lem:minimizer2} and Theorem~\ref{thm:equivalence2}, the real analogues of Lemma~\ref{lem:aux} and Theorem \ref{thm:equivalence}. We omit their proofs, which are very similar to the complex case.
\begin{lemma} \label{lem:minimizer2}
Let $\hat{A} \in \R^{n\times n}$ be given. Then, a solution of
\begin{equation} \label{auxmin2}
	\mathcal{T}(\hat{A}) = \min_{ \substack{T \in S(\Omega,n,\R) \\ \text{$T$ in modified real Schur form}}} \norm{\hat{A} - T}_F^2	
\end{equation}
is given by
\begin{equation} \label{blockT}
\mathcal{T}(\hat{A})_{IJ} = \left\{\begin{array}{lll}
    \hat{A}_{IJ} & I<J, & \text{(block upper triangular part)},\\
    p_{S(\Omega, 2, \mathbb{R})}(\hat{A}_{II}) & I=J, & \text{(block diagonal part)},\\
    0 & I > J & \text{(block lower triangular part)},\\
\end{array}\right.
\end{equation}
where $I,J \in \mathfrak{I}$ (as in~\eqref{partition}).
\end{lemma}
Analogously, we can define the block lower triangular matrix $\mathcal{L}(\hat{A}) = \hat{A}-\mathcal{T}(\hat{A})$, with blocks
\begin{equation} \label{blockL}
\mathcal{L}(\hat{A})_{IJ} = \left\{\begin{array}{lll}
    0 & I<J,\\
    p_{S(\Omega, 2, \mathbb{R})}(\hat{A}_{II}) & I=J,\\
    \hat{A}_{IJ} & I > J,\\
\end{array}\right.
\end{equation}
so that the optimum of~\eqref{auxmin2} is $\norm{\mathcal{L}(\hat{A})}_F^2$.

Similarly to the complex case, we can recast any (real) nearest $\Omega$-stable matrix problem as the optimization problem over the manifold of orthogonal matrices
\begin{equation} \label{minriemannreal}
	\min_{Q \in O(n)} \norm{\mathcal{L}(Q^\top AQ)}_F^2,
\end{equation}
where the function $\mathcal{L}$ is defined by either~\eqref{L} (if $\Omega \subseteq \R$) or~\eqref{blockL} (otherwise).

\begin{theorem} \phantom{A}\label{thm:equivalence2}
\begin{enumerate}
  \item The optimization problems~\eqref{minproblem} and~\eqref{minriemannreal} have the same minimum value.
  \item If $Q$ is a local (resp. global) minimizer for~\eqref{minriemannreal}, then $B = QTQ^\top $, where $T = \mathcal{T}(Q^\top AQ)$, is a local (resp. global) minimizer for~\eqref{minproblem}.
  \item If $B$ is a local (resp. global) minimizer for~\eqref{minproblem}, and $B=QTQ^\top $ is a (modified) Schur decomposition, then $Q$ is a local (resp. global) minimizer for~\eqref{minriemannreal} and $T = \mathcal{T}(Q^\top AQ)$.
\end{enumerate}
\end{theorem}
Moreover, it is clear that once again we can restrict the optimization problem to $Q \in SO(n)$. This concludes the theoretical overview needed to reformulate (complex or real) nearest $\Omega$-stable matrix problems as minimization over (unitary or orthogonal) matrix manifolds. 

To solve~\eqref{minriemann2} or~\eqref{minriemannreal} in practice for a given closed set $\Omega$, we need an implementation of the projection $p_\Omega$. In addition, for real problems with $\Omega \not \subseteq \R$, we also need an implementation of $p_{S(\Omega,2,\mathbb{R})}$ for $2\times 2$ matrices, which is not obvious. In the next section, we discuss how to develop such an implementation in the important cases $\Omega = \Omega_H$ (nearest Hurwitz stable matrix) and $\Omega = \Omega_S$ (nearest Schur stable matrix).

\section{Computing real $2\times 2$ nearest stable matrices}\label{sec:real2by2}

We start with a few auxiliary results that will be useful in this section. Given $\alpha \in \R$, we set
\[
U(\alpha) := \begin{bmatrix}
    \cos \alpha & -\sin \alpha\\
    \sin \alpha & \cos \alpha
\end{bmatrix}.
\]

\begin{lemma} \label{lem:equilibrate2}
Let $A\in \mathbb{R}^{2\times 2}$. Then, there exists $\alpha \in [0,\pi/2)$ such that $\hat{A} = U(\alpha)A U(\alpha)^\top$ has $\hat{A}_{1,1}=\hat{A}_{2,2} = \frac12 \trace(A)$.
\end{lemma}
\begin{proof}
Let $f(\alpha) = \hat{A}_{1,1}-\hat{A}_{2,2}$. If $A_{1,1}=A_{2,2}$ we can take $\alpha=0$. Otherwise, note that $f(0) = A_{1,1}-A_{2,2}$ and $f(\pi/2) = A_{2,2}-A_{1,1}$ have opposite signs. Hence, by continuity, there is $\alpha \in (0,\pi/2)$ such that $f(\alpha)=0$. Since $A$ and $\hat{A}$ are similar, they have the same trace, and hence $\trace(A) = \trace(\hat{A}) = 2\hat{A}_{1,1}$.
\end{proof}

\noindent We say that a non-empty closed set $\mathcal{S} \subseteq \mathbb{R}^{2\times 2}$ is \emph{rotation-invariant} if $X\in\mathcal{S}$ implies $U(\alpha)XU(\alpha)^\top \in \mathcal{S}$ for all $\alpha \in \R$.

\begin{lemma} \label{lem:rotinvariant}
Let $\mathcal{S} \subseteq \R^{2 \times 2}$ be rotation-invariant, and let $B$ be a local minimizer of
\[
\min_{X \in \mathcal{S}} \norm{A-X}_F
\]
for some $A\in \mathbb{R}^{2\times 2}$ satisfying $A_{1,1}=A_{2,2}$. Then:
\begin{enumerate}
    \item If $A_{2,1} \neq -A_{1,2}$, then $B_{1,1}=B_{2,2}$.
    \item If $A_{2,1} = -A_{1,2}$, then there exists $\alpha \in [0,\pi/2)$ such that $B = U(\alpha)\hat{B}U(\alpha)^\top$, where $\hat{B}$ is another local minimizer with the same objective value and $\hat{B}_{1,1} = \hat{B}_{2,2}$.
\end{enumerate}
\end{lemma}
\begin{proof}
1. Let $f(\alpha) = \norm{A - U(\alpha)B U(\alpha)^\top}_F^2$. A direct computation shows that
\[
\frac{df}{d\alpha}(0) = 2(A_{1,2}+A_{2,1})(B_{2,2}-B_{1,1}),
\]
Note that $U(\alpha)B U(\alpha)^\top \in \mathcal{S}$ for each $\alpha$. Hence, under our assumptions $f$ has a local minimum for $\alpha=0$. Thus the right-hand side must vanish, and this implies $B_{1,1}=B_{2,2}$.

2. $A$, $U(\alpha)$ and $U(\alpha)^\top$ are matrices of the form $\begin{bsmallmatrix}
    a & -b\\ b & a
\end{bsmallmatrix}$ for some $a,b \in \R$. Hence, they all commute with each other; in particular, for each $X\in \mathcal{S}$,
\begin{equation} \label{commut}
\norm{A - U(\alpha)^\top X U(\alpha)}_F = \norm{U(\alpha) A U(\alpha)^\top - X}_F = \norm{U(\alpha) U(\alpha)^\top A  - X}_F = \norm{A  - X}_F.
\end{equation}
This computation shows that if $B$ is a local minimizer, then all matrices of the form $U(\alpha)^\top B U(\alpha)$ are local minimizers with the same objective value.
By Lemma~\ref{lem:equilibrate2}, we can choose $\alpha$ so that $U(\alpha)^\top BU(\alpha) = \hat{B}$ has $\hat{B}_{1,1} = \hat{B}_{2,2}$.
\end{proof}

\noindent We say that a non-empty closed set $\mathcal{S} \subseteq \mathbb{R}^{2\times 2}$ is \emph{doubly rotation-invariant} if $X\in\mathcal{S}$ implies $U(\alpha)XU(\beta)^\top \in \mathcal{S}$ for any $\alpha,\beta \in \R$.

\begin{lemma} \label{lem:doublyrotinvariant}
Let $\mathcal{S} \subseteq \R^{2 \times 2}$ be doubly rotation-invariant, and let $B$ be a local minimizer of 
\[
\min_{X \in \mathcal{S}} \norm{A-X}_F
\]
for some diagonal $A = \operatorname{diag(\sigma_1,\sigma_2)} \in \R^{2 \times 2}$ with $\sigma_1 > 0$ and $\sigma_1 \geq \sigma_2 \geq 0$. Then:
\begin{enumerate}
    \item If $\sigma_1 \neq \sigma_2$, then $B$ is diagonal with $B_{11}-B_{22} \geq 0$ and $B_{11}+B_{22}\geq 0$.
    \item If $\sigma_1 = \sigma_2$, then there is $\alpha \in [0,\pi/2)$ such that $B = U(\alpha)\hat{B}U(\alpha)^\top$, where $\hat{B}$ is another local minimizer with the same objective value and $\hat{B}_{1,2}=\hat{B}_{2,1}=0$.
\end{enumerate}
\end{lemma}
\begin{proof}
1. We set $f(\alpha,\beta) = \norm{A - U(\alpha+\beta)B U(\alpha-\beta)^\top}_F^2$. Since $U(\alpha+\beta)B U(\alpha-\beta)^\top \in \mathcal{S}$ for all $\alpha,\beta\in \mathbb{R}$, $f$ must have a local minimum for $\alpha=\beta=0$. A direct computation shows that
\begin{subequations}
\begin{align}
\frac{\partial f}{\partial \alpha} (0,0) &= 2(\sigma_1-\sigma_2)(B_{21}+B_{12}), \\
\frac{\partial f}{\partial \beta} (0,0) &= 2(\sigma_1+\sigma_2)(B_{21}-B_{12}),  \label{diffbeta} \\
\frac{\partial^2 f}{\partial \alpha^2} (0,0) &= 4(\sigma_1-\sigma_2)(B_{11}-B_{22}), \\
\frac{\partial^2 f}{\partial \beta^2} (0,0) &= 4(\sigma_1+\sigma_2)(B_{11}+B_{22}).
\end{align}
\end{subequations}
Under our assumptions, $\sigma_1+\sigma_2 > 0$ and $\sigma_1-\sigma_2 > 0$, hence $B_{21}=B_{12}=0$, $B_{11}-B_{22} \geq 0$ and $B_{11}+B_{22}\geq 0$.

2. The matrix $A$ is in the form $\begin{bsmallmatrix}
    a & -b\\
    b & a
\end{bsmallmatrix}$, so~\eqref{commut} holds and shows that $\norm{A-B}_F = \norm{A-U(\alpha)BU(\alpha)^\top}_F$ for each $\alpha$. If $\sigma_1=\sigma_2 \neq 0$ then $\sigma_1+\sigma_2\neq 0$, and~\eqref{diffbeta} again shows that $B_{21}=B_{12}$. A direct computation shows then that $\hat{B} = U(\alpha)^\top BU(\alpha)$ has $\hat{B}_{12}=\hat{B}_{21}$ for each choice of $\alpha$. For $\alpha=0$ we have $\hat{B}_{12}=B_{12}$, while for $\alpha=\pi/2$ we have $\hat{B}_{12}=-B_{12}$. Hence by continuity there is $\alpha\in[0,\pi/2)$ such that $\hat{B}_{21}=\hat{B}_{12}=0$.
\end{proof}

The following result is a simple variant of the Routh--Hurwitz stability criterion (see e.g.~\cite[Sec. 26.2]{handbook}).
\begin{lemma} \label{lem:weakhurwitz}
Both roots of the polynomial $p(\lambda) = a\lambda^2 + b\lambda + c$, with $a \neq 0$ lie in the closed left half-plane $\Omega_H$ if and only if $a,b,c$ are either all non-negative or all non-positive.
\end{lemma}
Note that we can extend the result so that it holds also when $a=0$ by considering $\infty$ as a root belonging to $\Omega_H$.
\begin{proof}
We may assume (after a division) that $a=1$. If $p(\lambda)$ has two complex conjugate roots $\alpha \pm i \beta$, then stability is equivalent to $b = -2 \alpha \geq 0$, whereas $c=\alpha^2 + \beta^2 \geq 0$ is vacuously true. If instead $p$ has two real roots $\lambda_1$ and $\lambda_2$, then $\lambda_1 \leq 0$ and $\lambda_2 \leq 0$ if and only if \[\begin{cases}
b=-\lambda_1 - \lambda_2 \geq 0;\\
c=\lambda_1 \lambda_2 \geq 0.
\end{cases}  \]
\end{proof}

\subsection{Hurwitz stability}

By applying Lemma~\ref{lem:weakhurwitz} to the characteristic polynomial of $X \in \mathbb{R}^{2\times 2}$, it follows immediately that the set of $2 \times 2$ real Hurwitz stable matrices $S(\Omega_H, 2, \R)$ is $\mathcal{S}_t \cap \mathcal{S}_d$, where 
\[  \mathcal{S}_t : =  \{  X \in \R^{2 \times 2} : \trace(X) \leq 0   \}, \qquad  \mathcal{S}_d :=   \{ X \in \R^{2 \times 2} : \det(X) \geq 0.  \}    \]

Observe that the frontiers of $\mathcal{S}_t$ and $\mathcal{S}_d$ are, respectively, the sets of traceless and singular matrices
\[  \partial \mathcal{S}_t = \{  X \in \R^{2 \times 2} : \trace(X)=0  \}, \quad  \partial \mathcal{S}_d = \{ X \in \R^{2 \times 2} : \det(X)=0  \}.     \]

To visualize the geometry of $S(\Omega_H, 2, \R)$, we can assume up to a change of basis that $A_{11} = A_{22}$; indeed, the orthogonal change of basis in Lemma~\ref{lem:equilibrate2} preserves eigenvalues and Frobenius distance. Under this assumption, since $S(\Omega_H,2,\R)$ is rotation-invariant, Lemma~\ref{lem:rotinvariant} shows that (even when there are multiple minimizers) we can always choose a nearest Hurwitz stable matrix $B$ with $B_{11} = B_{22}$. Hence it makes sense to visualize $S(\Omega_H, 2, \R) \cap \{A \in\mathbb{R}^{2\times 2} \colon A_{11}=A_{22}\}$ as a volume parametrized by the three coordinates $A_{11},A_{12},A_{21}$. We show a few images of this set in Figure~\ref{fig:3d}.
\begin{figure}
\centering
\begin{tabular}{ccc}
$\vcenter{\hbox{\includegraphics[width = 0.33\textwidth]{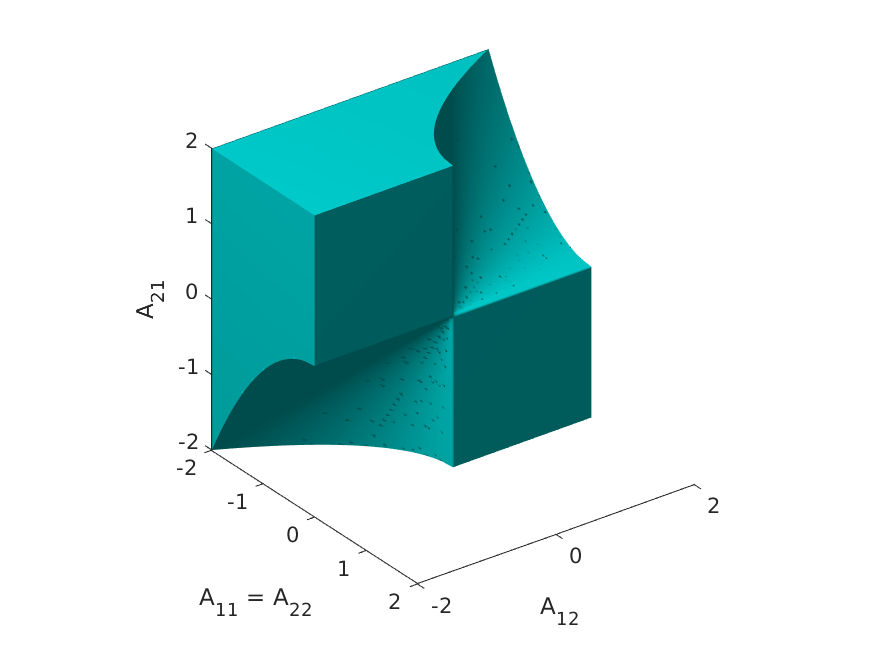}}}$&
$\vcenter{\hbox{\includegraphics[width = 0.33\textwidth]{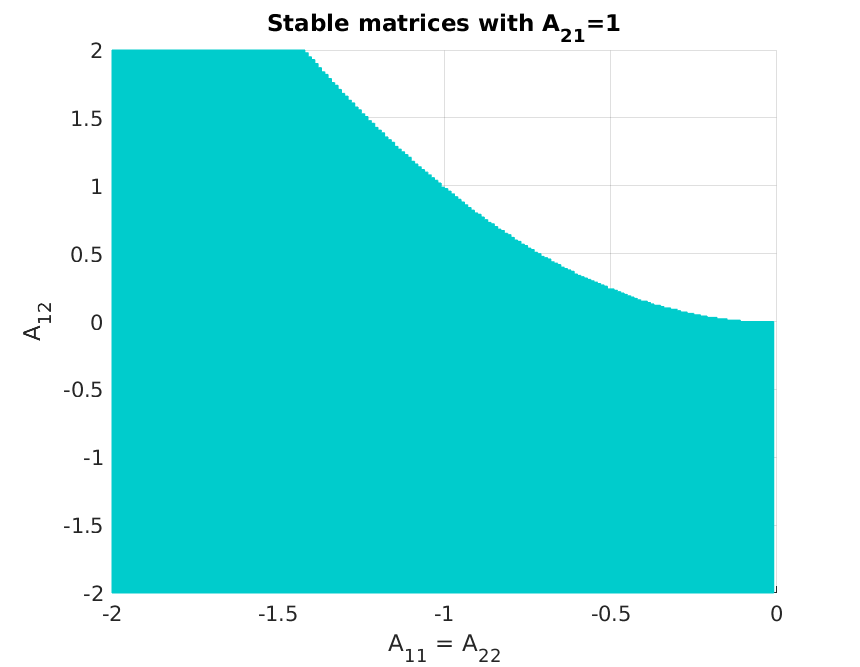}}}$&
$\vcenter{\hbox{\includegraphics[width = 0.33\textwidth]{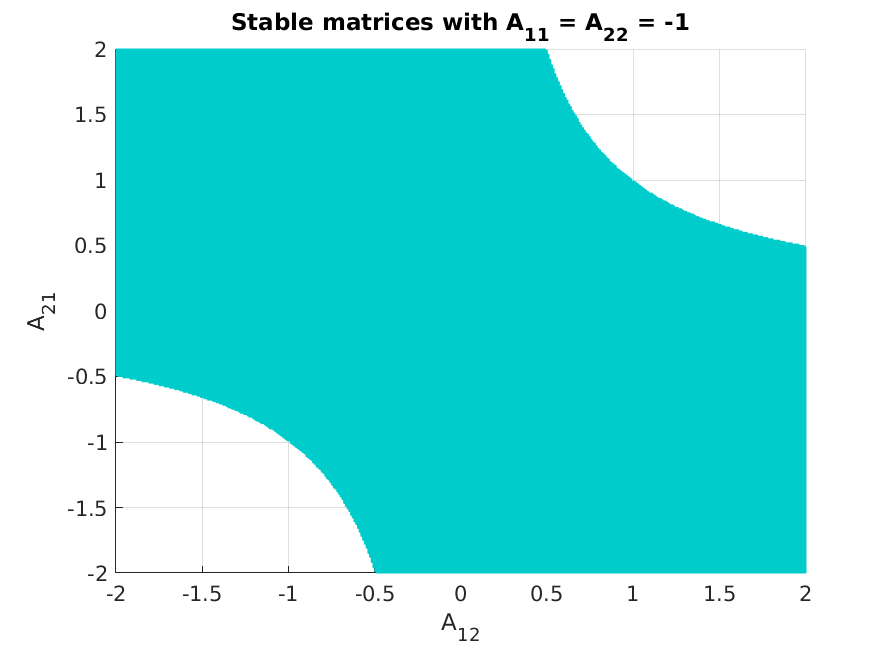}}}$
\end{tabular}
\caption{Some images of the set of Hurwitz stable matrices $S(\Omega_H, 2, \R)$ (in cyan), intersecated with the hyperplane $A_{11}=A_{22}$ and the cube $-2 \leq A_{ij} \leq 2$. The left picture is a 3D view of the set; the center and right one are two 2D sections with $A_{21}=1$ and $A_{11}=A_{22}=-1$, respectively.} \label{fig:3d}
\end{figure}
One can distinguish, in the 3D image on the left of Figure~\ref{fig:3d}, two straight faces (corresponding to the linear condition $\trace(A) = 0$) and two curved ones (corresponding to the nonlinear condition $\det(A) = 0$). The `cross' formed by the two edges $A_{11}=A_{22}=A_{21}=0$ and $A_{11}=A_{22}=A_{12}=0$ corresponds to matrices with two zero eigenvalues.

The following result can be used to find explicitly a nearest Hurwitz stable matrix to $A$.
\begin{lemma} \label{lem:2x2hurwitz}
Let $A\in \mathbb{R}^{2\times 2}$ have a singular value decomposition $A=U \begin{bsmallmatrix}
    \sigma_1\\ & \sigma_2
\end{bsmallmatrix}V^\top $, and $G\in SO(2)$ be a matrix such that $\hat{A}=G^\top AG$ satisfies $\hat{A}_{11}=\hat{A}_{22}$ ($G$ exists by Lemma~\ref{lem:equilibrate2}).

Then, the set
\begin{equation} \label{setofmins}
    \left\{A,\, A - \frac12 \trace(A) I_2, \,
    B_0,\, B_+,\, B_-\right\}
\end{equation}
with
\[
    B_0 = U \begin{bmatrix}
            \sigma_1 & 0\\ 0 & 0
        \end{bmatrix}V^\top 
\]
and
\begin{equation} \label{Bplusminus}
B_+ =       G
\begin{bmatrix}
    0 & \hat{A}_{12}\\
    0 & 0
\end{bmatrix}G^\top , \qquad
B_- = G\begin{bmatrix}
    0 & 0\\
    \hat{A}_{21} & 0
\end{bmatrix}G^\top ,
\end{equation}
contains a Hurwitz stable matrix nearest to $A$.
\end{lemma}

\begin{proof}
We assume that $A$ is not Hurwitz stable, otherwise the result is trivial, and let $B$ be a Hurwitz stable matrix nearest to $A$. One of the following three cases must hold (depending on which of the two constraints $\det(X) \geq 0$ and $\trace(X) \leq 0$ are active):
\begin{enumerate}
    \item $B \in \partial\mathcal{S}_t$, $B \not\in \partial\mathcal{S}_d$; 
    \item $B \in \partial\mathcal{S}_d$, $B \not\in \partial\mathcal{S}_t$, ;
    \item $B \in \partial\mathcal{S}_t \cap \partial\mathcal{S}_d$.
\end{enumerate}
Indeed, any Hurwitz stable matrix $B$ nearest to $A$ must belong to the boundary of the set. This implies that either $\det(B)=0$, or $\trace(B)=0$, or both. We treat the three cases separately.

\begin{enumerate}
    \item $B \in \partial\mathcal{S}_t$, $B \not\in \partial\mathcal{S}_d$. Then, $B$ must be a local minimizer of $\norm{A-X}_F^2$ in $\partial \mathcal{S}_t$. The only such minimizer is $A - \frac12 \trace(A) I$:
this is easy to see by parametrizing $X = \begin{bsmallmatrix}
    x & y\\
    z & -x
\end{bsmallmatrix}\in\partial \mathcal{S}_t$ and studying the resulting function, which is strictly convex trivariate quadratic.

\item $B \in \partial\mathcal{S}_d$, $B \not\in \partial\mathcal{S}_t$.
As in the previous case, note that $B$ must be a local minimizer of $\norm{A-X}_F^2$ in $\partial \mathcal{S}_d$. In addition, $\partial \mathcal{S}_d$ is doubly rotation-invariant, and we may assume $\sigma_1\neq 0$ since otherwise $A=0$, which is Hurwitz stable. We divide further into two subcases.
\begin{enumerate}
    \item $\sigma_1 > \sigma_2$. We have
\[
\norm{A-X}_F = \norm{U \begin{bsmallmatrix}
    \sigma_1 \\ & \sigma_2
\end{bsmallmatrix}V^\top - X}_F = \norm{\begin{bsmallmatrix}
    \sigma_1 \\ & \sigma_2
\end{bsmallmatrix} - U^\top XV}_F,
\]
hence, for a local minimizer $B \in \partial \mathcal{S}_d$, the matrix $C = U^\top BV$ is a local minimizer of $\norm{\begin{bsmallmatrix}
    \sigma_1 \\ & \sigma_2
\end{bsmallmatrix} - U^\top XV}_F$ in $U^\top(\partial \mathcal{S}_d) V = \partial \mathcal{S}_d$. In particular, Lemma~\ref{lem:doublyrotinvariant} shows that
\[
U^\top BV = \begin{bmatrix}
    \tau_1 \\ & \tau_2
\end{bmatrix}, \quad \tau_1, \tau_2 \in\mathbb{R}.
\]
For this matrix to be in $\partial \mathcal{S}_d$, we must have $\tau_1\tau_2=0$. Hence, we have a point $(\tau_1,\tau_2)\in\mathbb{R}^2$ on one of the two coordinate axes $\tau_1=0$ or $\tau_2=0$ which minimizes locally the distance from the point $(\sigma_1,\sigma_2)$ among all points on the axes. An easy geometric argument shows that the only such minimizers are $(\tau_1,\tau_2) = (0,\sigma_2)$ and $(\tau_1,\tau_2) = (\sigma_1,0)$, and only the latter satisfies the condition $\tau_1-\tau_2\geq 0$ coming from Lemma~\ref{lem:doublyrotinvariant}. Hence $U^\top BV = \begin{bsmallmatrix}
    \sigma_1\\ & 0
\end{bsmallmatrix}$, i.e., $B=B_0$.
    \item $\sigma_1 = \sigma_2$. We argue as in the previous case, but now from Lemma~\ref{lem:doublyrotinvariant} it follows only that $U^\top BV = U(\alpha) \begin{bsmallmatrix}
        \tau_1\\ & \tau_2
    \end{bsmallmatrix} U(\alpha)^\top$ ,
    where $\begin{bsmallmatrix}
        \tau_1\\ & \tau_2
    \end{bsmallmatrix}\in\partial \mathcal{S}_d$ is another local minimizer of $\norm{\begin{bsmallmatrix}
    \sigma_1\\ & \sigma_2
\end{bsmallmatrix} - X}_F$. By the same argument as in the previous case, this minimizer must be $\begin{bsmallmatrix}
        \tau_1\\ & \tau_2
    \end{bsmallmatrix} = \begin{bsmallmatrix}
    \sigma_1\\ & 0
\end{bsmallmatrix}$. Hence
\[
B = B_\alpha = U U(\alpha)  \begin{bsmallmatrix}
    \sigma_1\\ & 0
\end{bsmallmatrix} U(\alpha)^\top V^\top.
\]
for some $\alpha \in [0,\pi/2)$. The matrices $B_\alpha$ all have the same distance from $A$, but it may be the case that only some of them are Hurwitz stable\footnote{If $U,V \in SO(2)$ are rotation matrices, then $U(\alpha)$ commutes with $U$ and $V$ and hence $B_\alpha=U(\alpha)B_0U(\alpha)^\top$, so that the eigenvalues of $B_\alpha$ do not depend on $\alpha$. However, generally $U$ and $V$ may have determinant $-1$.}.
 In the case when $B_0$ is Hurwitz stable, $B_0$ is another Hurwitz stable matrix nearest to $A$, and we have proved the thesis. In the case when $B = B_\alpha \in \partial \mathcal{S}_d$ is Hurwitz stable but $B_0 \in \partial \mathcal{S}_d$ is not, by continuity there is a smallest value $\hat{\alpha}$ for which $B_\alpha$ is Hurwitz stable, and $B_{\hat{\alpha}}$ must belong to $\partial \mathcal{S}_t$. In particular, $B_{\hat{\alpha}}$ is another Hurwitz stable matrix nearest to $A$ in $\partial \mathcal{S}_t \cap \partial \mathcal{S}_d$, and by case 3 below $B_{\hat{\alpha}} \in \{B_+, B_-\}$.
\end{enumerate}
\item $B \in \partial\mathcal{S}_t \cap \partial\mathcal{S}_d$. Note that $\partial\mathcal{S}_t \cap \partial\mathcal{S}_d$ is the set of matrices with a double zero eigenvalue, and it is rotation-invariant. We have
\[
\norm{A-X}_F = \norm{G^\top AG - G^\top XG}_F = \norm{\hat{A} - G^\top XG}_F,
\]
hence $B$ is a global minimizer of $\norm{A-X}_F$ in $\partial \mathcal{S}_t \cap \partial \mathcal{S}_d$ if and only if $C = G^\top BG$ is a global minimizer of $\norm{\hat{A} - C}_F$ in $G^\top(\partial\mathcal{S}_t \cap \partial\mathcal{S}_d)G = \partial\mathcal{S}_t \cap \partial\mathcal{S}_d$. By Lemma~\ref{lem:rotinvariant}, we may assume (up to replacing $C$ with another global minimizer) that $C_{1,1}=C_{2,2}$, and since $\trace{C}=0$ it must be the case that $C_{1,1}=C_{2,2}=0$. Then, $\det C=0$ implies that either $C_{21}=0$ or $C_{12}=0$. To minimize $\norm{\hat{A}-C}_F$ under these constraints, the only remaining nonzero entry must be equal to the corresponding entry of $\hat{A}$. Hence $B = B_+$ or $B= B_-$.
\end{enumerate}
\end{proof}

Lemma~\ref{lem:2x2hurwitz} yields an explicit algorithm to find a projection $p_{S(\Omega_H,2,\mathbb{R})}(A)$: we compute the five matrices in the set~\eqref{setofmins}, and among those of them that are Hurwitz stable we choose one which is nearest to $A$. (Note that at least $B_+$ and $B_-$ are always Hurwitz stable, so this algorithm always returns a matrix.)

\begin{example} \label{example:generic}
Let $A = \begin{bmatrix}
    1 & 2\\ 1 & 1 
\end{bmatrix}$. Then,
    \begin{itemize}
        \item $A$ is not Hurwitz stable since it has an eigenvalue $1+\sqrt{2} > 0$.
        \item   $A - \frac12 \trace(A) I_2 = \begin{bmatrix}
    0 & 2\\
    1 & 0
\end{bmatrix}$ is not Hurwitz stable either since it has an eigenvalue $\sqrt{2} > 0$.
\item $B_0 \approx \begin{bmatrix}
    1.17 & 1.89\\
    0.72 & 1.17
\end{bmatrix}$, and this matrix is not Hurwitz stable either since it has a positive eigenvalue $\approx 2.34$.
\item $G=I$, $B_+ = \begin{bmatrix}
    0 & 2\\
    0 & 0
\end{bmatrix}$, and $B_- = \begin{bmatrix}
    0 & 0\\
    1 & 0
\end{bmatrix}$.
\end{itemize}
So the set~\eqref{setofmins} contains only two Hurwitz stable matrices, $B_+$ and $B_-$. We have $\norm{A-B_+}_F = \sqrt{3}$ and $\norm{A-B_-}_F = \sqrt{6}$. The smallest of these two values is achieved by $B_+$, hence by Lemma~\ref{lem:2x2hurwitz} the matrix $B_+$ is a Hurwitz stable matrix nearest to $A$, and we can set $p_{S(\Omega_H,2,\mathbb{R})}(A) = B_+$. Going through the proof of Lemma~\ref{lem:2x2hurwitz}, we can also show that this projection is unique. Note that, by continuity, for any matrix in a sufficiently small neighborhood of $A$ the same inequalities will hold, and hence we will fall under the same case. In particular, the set of matrices $A$ for which $p_{S(\Omega_H,2,\mathbb{R})}(A)$ has a double zero eigenvalue is not negligible.
\end{example}

\begin{remark}
After the present paper appeared as a preprint, a follow-up work has appeared in which the authors find a different formula for the minimizer that does not require comparing several candidates~\cite{KuoEtAl}.
\end{remark}

\subsection{Schur stability}

We can obtain analogous results for Schur stability.

\begin{lemma}
We have $S(\Omega_S,2,\mathbb{R}) = \mathcal{S}_{-} \cap \mathcal{S}_0 \cap \mathcal{S}_{+}$, where
\begin{subequations} \label{schur-conditions}
\begin{align}
\mathcal{S}_0 &= \{X \in \mathbb{R}^{2\times 2} \colon \det(X) \leq 1 \}, \\
\mathcal{S}_{\pm} &= \{X \in \mathbb{R}^{2\times 2} \colon \pm \trace(X) \leq 1 + \det(X) \}.
\end{align}
\end{subequations}
\end{lemma}
\begin{proof}
Let $t = \trace X$ and $d = \det X$ for brevity. The matrix $X$ is Schur stable if and only if its characteristic polynomial $p(\lambda) = \lambda^2 - t\lambda + d$ has both roots inside the closed unit disk. This is equivalent to imposing that its Cayley transform~\cite{HinPbook}, 
\[
q(\mu) := (\mu-1)^2p\left(\frac{\mu+1}{\mu-1}\right) = (1-t+d)\mu^2 + 2(1-d)\mu + (1+t+d),
\]
has both roots inside the (closed) left half-plane. (Note that this result extends to the case when $\lambda=1$ is a root, corresponding to $\mu=\infty$, which is by convention to be considered in the closed left half-plane.) By Lemma~\ref{lem:weakhurwitz}, this holds if and only if the coefficients $1-t+d, 2(1-d), 1+t+d$ are all non-negative or all non-positive. We shall show that these three coefficients cannot all be non-positive at the same time: indeed, if the three relations $1-t+d \leq 0, 1-d\leq 0$ and $1+t+d\leq 0$ hold, then one gets
\[
2 \leq 1+d \leq t \leq -1-d \leq -2,
\]
which is impossible. Hence Schur stability is equivalent to the three coefficients being non-negative, which are the conditions~\eqref{schur-conditions}.
\end{proof}

We note in passing that $X \in \mathcal{S}_+ \cap \mathcal{S}_-$ implies in turn that $-1 -\det(X) \leq 1+ \det(X) \Leftrightarrow \det(X) \geq -1$. Exactly like for Hurwitz stable matrices, one can assume (up to an orthogonal similarity) that $A$ has $A_{11}=A_{22}$, and in this case Lemma~\ref{lem:rotinvariant} ensures that there is a nearest Schur stable matrix $B$ with $B_{11}=B_{22}$. Hence it makes sense to visualize in 3D space the set $S(\Omega_S, 2, \R) \cap \{A \in\mathbb{R}^{2\times 2} \colon A_{11}=A_{22}\}$. A few images of this set are shown in Figure~\ref{fig:3dschur}.
\begin{figure}
\centering
\begin{tabular}{ccc}
$\vcenter{\hbox{\includegraphics[width = 0.33\textwidth]{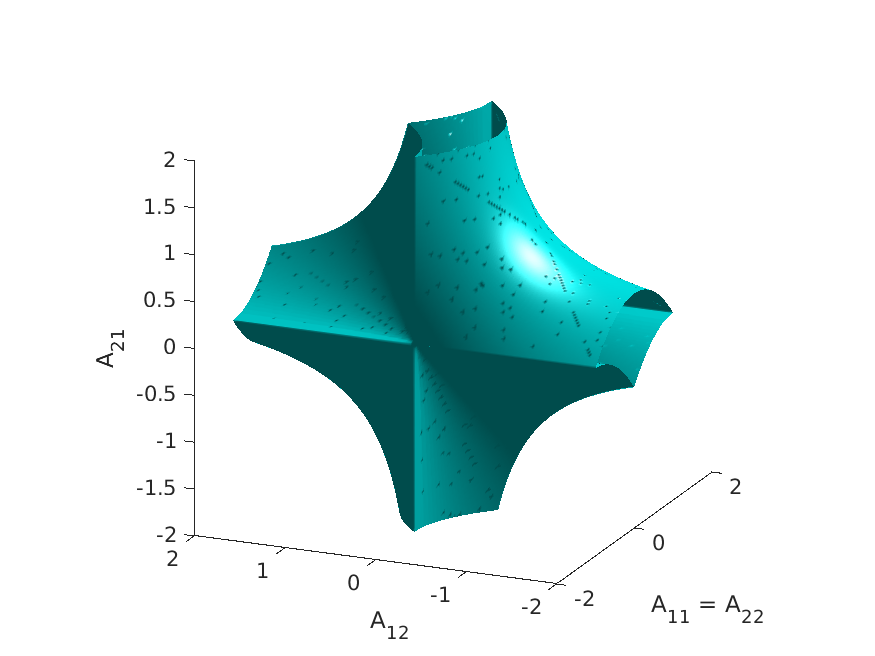}}}$&
$\vcenter{\hbox{\includegraphics[width = 0.33\textwidth]{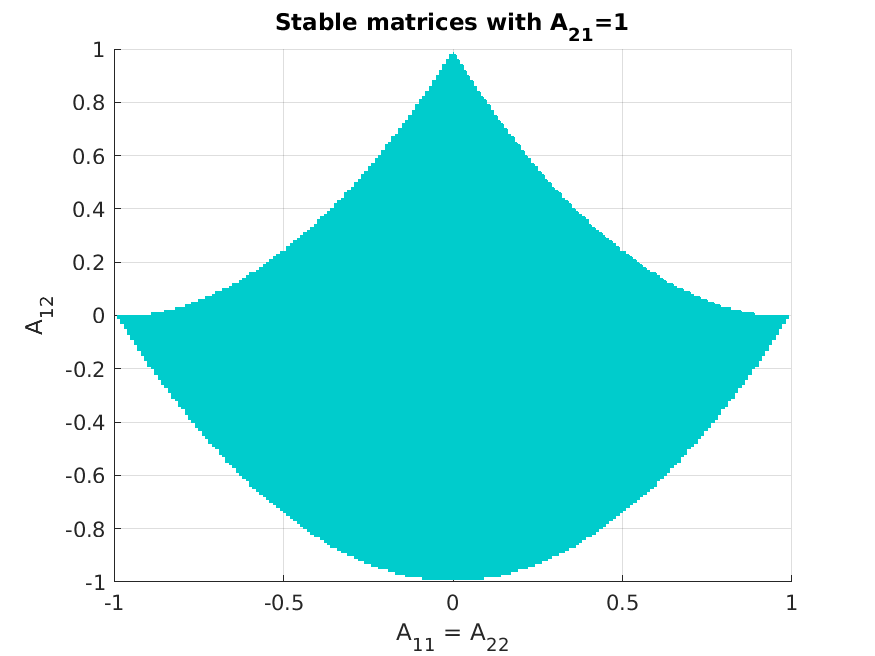}}}$&
$\vcenter{\hbox{\includegraphics[width = 0.33\textwidth]{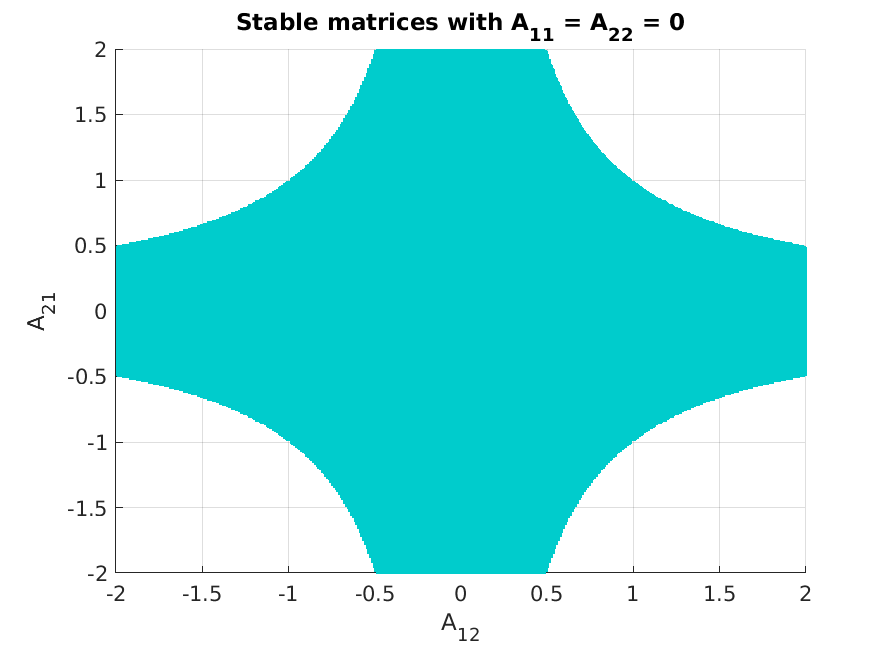}}}$
\end{tabular}
\caption{Some images of the set of Schur stable matrices $S(\Omega_S, 2, \R)$ (in cyan), intersecated with the hyperplane $A_{11}=A_{22}$ and the cube $-2 \leq A_{ij} \leq 2$. The left picture is a 3D view of the set; the center and right one are two 2D sections with $A_{21}=1$ and $A_{11}=A_{22}=0$, respectively.} \label{fig:3dschur}
\end{figure}
Note the cross visible in the front of the figure, formed by the matrices with $A_{11}=A_{22}=-1$ and either $A_{12}=0$ or $A_{21}=0$; these matrices all have a double eigenvalue $-1$. There is an analogous cross in the back of the figure with $A_{11}=A_{22}=1$. In addition, on the top-left and bottom-right side of the 3D figure there are two sharp edges which correspond to the matrices with $A_{11}=A_{22}=0$ and $A_{21}A_{12}=1$. These edges are formed by the intersection of the two smooth surfaces $\mathcal{S}_+$ and $\mathcal{S}_-$, and are formed by points where two constraints are active simultaneously. Observe that there are no corresponding sharp edges in the other two quadrants, since only one constraint is active for the matrices with $A_{11}=A_{22}=0$ and $A_{21}A_{12}=-1$.

To formulate the analogue of Lemma~\ref{lem:2x2hurwitz} for Schur stable matrices, we first define $\mathcal{M}(\sigma_1,\sigma_2)$ to be the set of critical points $(\tau_1,\tau_2)\in\mathbb{R}^2$ of the function
\[
(\tau_1-\sigma_1)^2 + (\tau_2-\sigma_2)^2
\]
subject to $\tau_1 \tau_2=1$. Geometrically, these are the critical points of the distance between a given point $(\sigma_1, \sigma_2)\in\mathbb{R}^2$ and the hyperbola $\tau_1 \tau_2 = 1$. The solutions to this problem can be computed exactly for any given pair $(\sigma_1,\sigma_2)$, since solving
\[
\frac{d}{dt} \left((\sigma_1 - t)^2 + \left(\sigma_2 - \frac{1}{t}\right)^2\right) = 0
\]
amounts to computing the roots of a degree-4 polynomial. In particular, the set $\mathcal{M}(\sigma_1,\sigma_2)$ has at most four elements for any choice of $(\sigma_1,\sigma_2)$.

With this definition, we can state the following lemma.
\begin{lemma} \label{lem:schurmin}
Let $A\in \mathbb{R}^{2\times 2}$, and $G\in SO(2)$ be a matrix such that $\hat{A}=G^\top AG$ satisfies $\hat{A}_{11}=\hat{A}_{22}$ ($G$ exists by Lemma~\ref{lem:equilibrate2}).

Then, the set
\begin{equation} \label{setofschurmins}
    \left\{A,\, B_+, \, B_- \right \} \cup \mathcal{B}_0 \cup \mathcal{B}_{+} \cup \mathcal{B}_{-}  \cup \mathcal{B}_{*} 
\end{equation}
where 
\[
\mathcal{B}_0 = \left\{U_0 \begin{bmatrix}
    \tau_1\\ & \tau_2
\end{bmatrix}V_0^\top \colon (\tau_1,\tau_2) \in \mathcal{M}(\sigma_{0,1},\sigma_{0,2}) \right\} \quad \text{with an SVD $A = U_0 \begin{bmatrix}
    \sigma_{0,1}\\ & \sigma_{0,2}
\end{bmatrix}V_0^\top$},
\]
\[
B_{\pm} = \pm I + U_{\pm} \begin{bmatrix}
    \sigma_{\pm,1}\\  & 0
\end{bmatrix} V_{\pm} \quad \text{with an SVD $A \mp I = U_{\pm} \begin{bmatrix}
    \sigma_{\pm,1}\\ & \sigma_{\pm,2}
\end{bmatrix} V_{\pm}^\top$},
\]
\[
\mathcal{B}_{\pm} = \left\{G\begin{bmatrix}
    \pm 1 & \hat{A}_{1,2}\\ 0 & \pm 1
\end{bmatrix}G^\top, G\begin{bmatrix}
    \pm 1 & 0\\ \hat{A}_{2,1} & \pm 1
\end{bmatrix}G^\top  \right\},
\]
and
\[
\mathcal{B}_{*} = \left\{G \begin{bmatrix}
    0 & \tau_1 \\
    \tau_2 & 0
\end{bmatrix} G^\top \colon (\tau_1,\tau_2) \in \mathcal{M}(\hat{A}_{1,2}, \hat{A}_{2,1}) \right\}.
\]
contains a Schur stable matrix nearest to $A$.
\end{lemma}
Note that the set in~\eqref{setofschurmins} contains at most $15$ elements.
\begin{proof}
The proof of this result is similar to the one of Lemma~\ref{lem:2x2hurwitz} but more involved. For this reason, the arguments that are analogous to those already discussed for Lemma \ref{lem:2x2hurwitz} will only be sketched.

We assume that $A$ is not Schur stable, otherwise the result is trivial, and let $B$ be a Schur stable matrix nearest to $A$. 
Any Schur stable matrix $B$ nearest to $A$ must belong to the boundary of the set $\mathcal{S}_{-} \cap \mathcal{S}_0 \cap \mathcal{S}_{+}$. Observe that $\partial \mathcal{S}_0$ is the set of matrices with determinant $1$, while $\partial \mathcal{S}_{\pm}$ is the set of matrices with an eigenvalue $\pm 1$. In particular, the set $\partial \mathcal{S}_{-} \cap \partial \mathcal{S}_{0} \cap \partial \mathcal{S}_{+}$ where all three constraints are active is empty, since a $2\times 2$ matrix with both $1$ and $-1$ as eigenvalues cannot have determinant $1$. Hence, one of the following cases must hold (corresponding to all possible remaining combinations of active constraints):
\begin{enumerate}
    \item $B \in \partial\mathcal{S}_0$, $B \not\in \partial\mathcal{S}_{+}$, $B \not\in \partial\mathcal{S}_{-}$; 
    \item $B \in \partial\mathcal{S}_{+}$, $B \not\in \partial\mathcal{S}_{0}$, $B \not\in \partial\mathcal{S}_{-}$; 
    \item $B \in \partial\mathcal{S}_{-}$, $B \not\in \partial\mathcal{S}_{0}$, $B \not\in \partial\mathcal{S}_{+}$;
    \item $B \in \partial\mathcal{S}_0 \cap \partial\mathcal{S}_{+}$, $B \not\in \partial\mathcal{S}_{-}$;
    \item $B \in \partial\mathcal{S}_0 \cap \partial\mathcal{S}_{-}$, $B \not\in \partial\mathcal{S}_{+}$;
    \item $B \in \partial\mathcal{S}_{+} \cap \partial\mathcal{S}_{-}$, $B \not\in \partial\mathcal{S}_{0}$;
\end{enumerate}
We treat the six cases separately.
\begin{enumerate}
    \item $B \in \partial\mathcal{S}_0$, $B \not\in \partial\mathcal{S}_{+}$, $B \not\in \partial\mathcal{S}_{-}$. Then, $B$ is a local minimizer of $\norm{A-X}_F$ in the doubly rotation-invariant set $\partial\mathcal{S}_0$, and $C = U_0^\top B V_0$ is a local minimizer of $\norm{\begin{bsmallmatrix}
        \sigma_{0,1} \\ & \sigma_{0,2}
    \end{bsmallmatrix} - X}_F$ in $\partial \mathcal{S}_0$. 

    If $\sigma_{0,1}\neq \sigma_{0,2}$, then $C$ is diagonal by Lemma~\ref{lem:doublyrotinvariant}, and for $C = \begin{bsmallmatrix}
        \tau_1 \\ & \tau_2
    \end{bsmallmatrix}$ (with $\tau_1\tau_2 = \det C = 1$) to be a local minimizer we must have $(\tau_1,\tau_2) \in \mathcal{M}(\sigma_{0,1},\sigma_{0,2})$. Hence, $B \in \mathcal{B}_0$. 

    If $\sigma_{0,1} = \sigma_{0,2}$ instead, then $C = U(\alpha)\begin{bsmallmatrix}
        \tau_1 \\ & \tau_2
    \end{bsmallmatrix} U(\alpha)^\top$ and $B = B_\alpha = U_0 U(\alpha)\begin{bsmallmatrix}
        \tau_1 \\ & \tau_2
    \end{bsmallmatrix} U(\alpha)^\top V_0^\top$ for some $\alpha \in [0,\pi/2)$. If $B_0$ is stable, then $B_0\in \mathcal{B}_0$ is another nearest Schur stable matrix to $A$, and we are done. If $B_0$ is not stable, then there is a minimum $\alpha$ such that $B_\alpha$ is stable, and this is another nearest Schur stable matrix to $A$ in which one more constraint is active; hence it falls in one of the cases 4--6.

    \item $B \in \partial\mathcal{S}_{+}$, $B \not\in \partial\mathcal{S}_{0}$, $B \not\in \partial\mathcal{S}_{-}$. Then,
    \[
    \norm{A-X}_F = \norm{(U_{+}^\top(A - I)V_{+} - U_{+}^\top(X - I)V_{+}}_F,
    \]
    hence $B$ is a local minimizer of $\norm{A-X}_F$ in $\partial \mathcal{S}_+$ if and only if $C = U_{+}^\top(X - I)V_{+}$ is a local minimizer of $\norm{(U_{+}^\top(A - I)V_{+} - X}_F$ in $\partial \mathcal{S}_d$: indeed, $X$ has an eigenvalue $1$ if and only if $U_+^\top(X-I)V_+$ has determinant zero. Arguing as in Case 2 of Lemma~\ref{lem:2x2hurwitz}, either $C = \begin{bsmallmatrix}
        \sigma_{+,1} \\  & 0
    \end{bsmallmatrix}$ and hence $B = B_+$, or we can find a minimizer with the same objective value in one of the cases 4--6.
    \item $B \in \partial\mathcal{S}_{-}$, $B \not\in \partial\mathcal{S}_{0}$, $B \not\in \partial\mathcal{S}_{+}$. We argue as in Case 2, swapping all plus and minus signs.
    \item $B \in \partial\mathcal{S}_0 \cap \partial\mathcal{S}_{+}$, $B \not\in \partial\mathcal{S}_{-}$. Note that $\partial\mathcal{S}_0 \cap \partial\mathcal{S}_{+}$ is the (rotation-invariant) set of matrices with a double eigenvalue $+1$ (since they must have an eigenvalue equal to $+1$ and determinant $1$). $B$ is a minimizer if and only if $C= G^\top BG$ is a minimizer of $\norm{\hat A - X}_F$ in $\partial\mathcal{S}_0 \cap \partial\mathcal{S}_{+}$, and by Lemma~\ref{lem:rotinvariant} we can assume $C_{1,1}=C_{2,2}=\frac12 \trace (C) = \frac12 \trace (B) = +1$. Since $\det C = 1$, at least one among $C_{1,2}$ and $C_{2,1}$ is zero, and to minimize the distance from $\hat{A}$ the other must be equal to the corresponding element of $\hat{A}$. Thus we get $B \in \mathcal{B}_{+}$.
    \item $B \in \partial\mathcal{S}_0 \cap \partial\mathcal{S}_{-}$, $B \not\in \partial\mathcal{S}_{+}$. We argue as in Case 2, swapping all plus and minus signs.
    \item $B \in \partial\mathcal{S}_{+} \cap \partial\mathcal{S}_{-}$, $B \not\in \partial\mathcal{S}_{0}$. Note that $\partial\mathcal{S}_{+} \cap \partial\mathcal{S}_{-}$ is the (rotation-invariant) set of matrices with eigenvalues $1$ and $-1$. Arguing as in Case 4,
    \[
        G^\top BG = C = \begin{bmatrix}
            0 & \tau_1\\
            \tau_2 & 0
        \end{bmatrix}, \quad \tau_1\tau_2 = 1,
    \]
    as $C$ must have $C_{1,1}=C_{2,2}=\frac12 \trace C=0$ and $\det C = -1$. Moreover, to minimize $\norm{\hat{A}-C}_F$ we must have $(\tau_1,\tau_2) \in \mathcal{M}(\hat{A}_{1,2},\hat{A}_{2,1})$, and $B \in \mathcal{B}_{*}$.
\end{enumerate}

\end{proof}

\section{The gradient of the objective function} \label{sec:gradient}

To apply most optimization algorithms, one needs the gradient of the objective function; we compute it in this section. We start from the real case, which is simpler.

\subsection{Computing the gradient: real case} \label{sec:gradient-real}

We assume in this section that $\F = \R$; we wish to compute the gradient of the function $f(Q) = \norm{\mathcal{L}(Q^\top AQ)}_F^2$, where the function $\mathcal{L}(\cdot)$ is given by either~\eqref{L} or~\eqref{blockL}.

In the set-up of optimization on matrix manifolds, the most appropriate version of the gradient to consider is the so-called \emph{Riemannian gradient} \cite[Section~3.6]{AMS08}. In the case of an embedded manifold (which includes our case, since $O(n)$ is embedded in $\mathbb{R}^{n\times n} \simeq \mathbb{R}^{n^2}$) the Riemannian gradient $\grad f$ is simply the projection on the tangent space $T_Q O(n)$ of the Euclidean gradient of $f$: i.e., its gradient $\nabla_Q f$ when $f$ is considered as a function in the ambient space $\mathbb{R}^{n\times n} \to \mathbb{R}$.

Thus we start by computing the Euclidean gradient $\nabla_Q f$. The function $f = g \circ h$ is the composition of $g(X) = \norm{\mathcal{L}(X)}_F^2$ and $h(Q) = Q^\top AQ$. Here and in the following, for ease of notation, we set $L := \mathcal{L}(Q^\top AQ)$, together with $T := \mathcal{T}(Q^\top AQ)$ and $\hat{A}=Q^\top AQ = L+T$. The gradient of $g$ is
\[
\nabla_X g = 2 \mathcal{L}(X):
\]
this follows from the fact that when $i \neq j$ the gradient of $L_{ij}^2$ is $2L_{ij}$, and by Theorem~\ref{thm:distance-function} the gradient of $L_{ii}^2 = \abs{X_{ii} - p_\Omega(X_{ii})}^2$ is $2(X_{ii} - p_\Omega(X_{ii}))$. 

When one uses the block version~\eqref{blockL} of the function $\mathcal{L}(\cdot)$, the same result holds blockwise: the gradient of $\norm{L_{II}}_F^2 = \norm{X_{II} - p_{S(\Omega,2,\mathbb{R})}(X_{II})}_F^2$ is $2(X_{II} - p_{S(\Omega,2,\mathbb{R})}(X_{II}))$ even when $X_{II}$ is a $2\times 2$ block. This follows again from Theorem~\ref{thm:distance-function}, applied to the closed set $S(\Omega,2,\R) \subset \R^{2\times 2} \simeq \R^4$.

The Fréchet derivative of $h(Q)$ is $D_h[Q](H) = H^\top AQ + Q^\top AH$; hence, using vectorization and Kronecker products~\cite[Section~11.4]{handbook} its Jacobian is

\begin{equation*}
	J_h (\vvec Q) = ((AQ)^\top  \otimes I)\Pi + I \otimes Q^\top A,	
\end{equation*}
where the permutation matrix $\Pi=\Pi^\top $ is often called the \emph{vec-permutation matrix} \cite{ANT19} or the \emph{perfect shuffle matrix}, and it is defined as the $n^2\times n^2$ matrix such that $\Pi \operatorname{vec}(X) = \operatorname{vec} (X^\top) $ for all $X$.

By the chain rule,
\[
(\vvec \nabla_Q f)^\top  = (\vvec \nabla_{Q^\top AQ} g)^\top  J_h (\vvec Q),
\]
and hence, transposing everything,
\begin{align*}
\vvec \nabla_Q f &= \left(\Pi (AQ \otimes I) + I \otimes (A^\top Q)\right) 2\operatorname{vec} L
\\&= 2\operatorname{vec} (AQ L^\top  + A^\top Q L).
\end{align*}

We now need to project $\nabla_Q f = 2(AQ L^\top  + A^\top Q L)$ on the tangent space $T_Q O(n)$ to get the Riemannian gradient. One has (see~\cite[Example~3.5.3]{AMS08} or more generally \cite[Lemma 3.2]{ANT19})
\[
T_Q O(n) = \{QS \colon S = -S^\top  \}.
\]
In particular, we can write the projection of a matrix $M$ onto $T_Q O(n)$ by using the skew-symmetric part operator $\sk(G) = \frac12 (G-G^\top )$ as
\[
P_{T_Q O(n)}(M) = Q\sk(Q^\top M).
\]
Thus, recalling $\hat{A}=Q^\top AQ=L+T$,
\begin{align}
\grad f(Q) &= P_{T_Q O(n)}(\nabla_Q f)  \nonumber
\\&= Q \operatorname{skew}(Q^\top  \nabla_Q f)  \nonumber
\\&= 2Q \operatorname{skew}(\hat{A}L^\top  + \hat{A}^\top L)  \nonumber
\\&= 2Q \operatorname{skew}((L+T)L^\top  + (L^\top +T^\top )L)  \nonumber
\\&= 2Q \operatorname{skew}(TL^\top  + T^\top L) \nonumber
\\&= 2Q \operatorname{skew}(TL^\top  - L^\top T). \label{gradf-real}
\end{align}
The matrix $TL^\top  - L^\top T$ is a strictly upper triangular matrix with entries
\[
(TL^\top  - L^\top T)_{ij} = \sum_{k=i}^{j-1}\hat{A}_{ik}\hat{A}_{jk} - \sum_{k=i+1}^j \hat{A}_{ki}\hat{A}_{kj}, \quad i<j.
\]
\begin{remark} \label{rem:gradvanishes-real}
Observe that 
 the gradient vanishes if and only if $TL^\top = L^\top T$. It follows from the definitions of $T$ and $L$ that $Q T Q^\top = B$ and $Q L Q^\top = A-B$, hence changing basis this condition becomes $B(A-B)^\top = (A-B)^\top A$.
\end{remark}

\subsection{Computing the gradient: complex case} \label{sec:gradient-complex}

The computation of the gradient in the complex case is similar, but technically more involved. To work only with real differentiation, following (with a slightly different notation) \cite[Section 2.1]{ANT19}, we define a complex version of the vectorization operator to map $\C^{n\times n}$ into $\R^{2n^2}$
\[
\operatorname{cvec}(X) = \operatorname{vec} \begin{bmatrix}
\Re X & \Im X 
\end{bmatrix} = \begin{bmatrix}
    \operatorname{vec} \Re X\\
    \operatorname{vec} \Im X\\
\end{bmatrix}, \quad \operatorname{cvec}: \mathbb{C}^{n\times n} \to \mathbb{R}^{2n^2}.
\]
Here $\Re X$ and $\Im X$ denote the real and imaginary parts (defined componentwise) of the matrix $X$.

Similarly, we define a complex version of the Kronecker product $\otimes_c$, to obtain a complex version of the identity $\vvec(AXB) = B^\top \otimes A \vvec(X)$.
\begin{equation*}
\operatorname{cvec}(AXB) = \underbrace{\begin{bmatrix}
    \Re B^\top  \otimes \Re A - \Im B^\top  \otimes \Im A & -\Re B^\top  \otimes \Im A - \Im B^\top  \otimes \Re A\\
    \Re B^\top  \otimes \Im A + \Im B^\top  \otimes \Re A & \Re B^\top  \otimes \Re A - \Im B^\top  \otimes \Im A
\end{bmatrix}}_{=:B^* \mathbin{\otimes_c} A}\operatorname{cvec}(X).
\end{equation*}
Note that $(B^* \otimes_c A)^\top  = (B \otimes_c A^*)$.

Finally, let $\Pi_c\in\mathbb{R}^{2n^2\times 2n^2} = \operatorname{diag}(\Pi, -\Pi)$ be the permutation matrix (with $\Pi_c = \Pi_c^\top $) such that $\Pi_c \operatorname{cvec}(X) = \operatorname{cvec}(X^*)$. We now have all the complex vectorization machinery available to perform a direct analogue of the computation in the previous section. Again, for ease of notation we define $\hat{A} = U^\top AU=L+T$, with $L = \mathcal{L}(\hat{A})$ and $T = \mathcal{T}(\hat{A})$.

The gradient of $g(X) = \norm{\mathcal{L}(X)}_F^2$ is $\nabla_X g = 2 \mathcal{L}(X)$. The Jacobian of $h(U) = U^*AU$ is
\[
J_h(U) = ((AU)^* \otimes_c I)\Pi_c + I \otimes_c (U^*A).
\]
Hence
\begin{align*}
\operatorname{cvec} \nabla_U f &= (J_h(U))^\top \operatorname{cvec} \nabla_{\hat{A}} g
\\&= \left( \Pi_c ((AU) \otimes_c I) + I\otimes_c A^*U \right) 2\operatorname{cvec}(L)
\\&= 2\operatorname{cvec}(AUL^*+A^*UL).
\end{align*}

The tangent space to the manifold of unitary matrices is, by a special case of \cite[Lemma 3.2]{ANT19},
\[
T_U U(n) = \{US: S = -S^*\},
\]
and the associated projection is 
\[
P_{T_U U(n)}(M) = U \sk (U^*M), \quad \sk(X) = \frac12(X-X^*).
\]
Thus
\begin{align}
\grad f(U) &= P_{T_U U(n)}(\nabla_U f) \nonumber
\\&= 2U\operatorname{skew}(\hat{A}L^* + \hat{A}^*L) \nonumber
\\&= 2U\operatorname{skew}((L+T)L^* + (L^*+T^*)L) \nonumber
\\&= 2U\operatorname{skew}(TL^* + T^*L) \nonumber
\\&= 2U\operatorname{skew}(TL^* - L^*T). \label{gradf-complex}
\end{align}
The diagonal entries of $TL^* - L^*T$ are equal to $T_{ii}\overline{L_{ii}} - \overline{L_{ii}}T_{ii} =0$, hence that matrix is again strictly upper triangular. Its nonzero entries are given by
\[
(TL^* - L^*T)_{ij} = 
T_{ii}\overline{\hat{A}_{ji}} + \sum_{k=i+1}^{j-1} \hat{A}_{ik}\overline{\hat{A}_{jk}} + \hat{A}_{ij}\overline{L_{jj}} - \overline{L_{ii}}\hat{A}_{ij} - \sum_{k=i+1}^{j-1} \overline{\hat{A}_{ki}}\hat{A}_{kj} - \overline{\hat{A}_{ji}}T_{jj}, \quad i<j.
\]

\begin{remark} \label{rem:gradvanishes-complex}
By an argument analogous to that in Remark~\ref{rem:gradvanishes-real}, $\grad f(U)$ vanishes if and only if $B(A-B)^* = (A-B)^*B$.
\end{remark}

\section{Numerical experiments}\label{sec:numexp}

We used the solver \texttt{trustregions} from the Matlab package Manopt~\cite{BMAS14} (version 6.0) for optimization on matrix manifolds. The solver is a quasi-Newton trust-region method; a suitable approximation of the Hessian is produced automatically by Manopt using finite differences.

For problems with $\F=\R$, we specified the manifold $O(n)$, the cost function $f(Q) = \norm{\mathcal{L}(Q^\top AQ)}_F^2$ (with $\mathcal{L}(\cdot)$ as in~\eqref{blockL}) and the Riemannian gradient $\operatorname{grad} f(Q) = 2Q \operatorname{skew}(TL^\top -L^\top  T)$. For problems with $\F=\C$, we specified the manifold $U(n)$, the cost function $f(Q) = \norm{\mathcal{L}(Q^\top AQ)}_F^2$ (with $\mathcal{L}(\cdot)$ as in~\eqref{L}) and the Riemannian gradient $\operatorname{grad} f(U) = 2U \operatorname{skew}(TL^* -L^* T)$.

All the experiments were performed on an Intel i7-4790K CPU \@ 4.00~GHz with 32~Gib of RAM, running Matlab R2018b (Lapack 3.7.0 and MKL 2018.0.1) on a Linux system. Our implementation of the algorithm for the three cases $\Omega \in \{\Omega_H, \Omega_S, \mathbb{R}\}$ is available for download at \url{https://github.com/fph/nearest-omega-stable}.

\begin{remark}
It is somewhat surprising that our algorithm does not make use of an eigensolver (Matlab's \texttt{eig}, \texttt{schur}, \texttt{eigs} or similar), even though the original problem~\eqref{minproblem} is intrinsically about eigenvalues. Indeed, the minimization procedure itself is as an eigensolver in a certain sense: it computes a Schur form by trying to minimize the strictly subdiagonal part of $Q^\top AQ$, not much unlike Jacobi's method for eigenvalues  (\cite{Gre55-nonsymmetric-jacobi}, see also~\cite{Meh08-nonsymmetric-jacobi} for more references). Indeed, when run with $\Omega = \C$, the method essentially becomes a Jacobi-like method to compute the Schur form.
\end{remark}

\subsection{Experiments from Gillis and Sharma} \label{hurwitz-experiments}

We took matrices with dimension $n\in \{10,20,50,100\}$ of six different types as follows.
\begin{center}
\begin{tabular}{rl}
\toprule
Type & Description\\
\midrule
1 & $A_{ij}= \begin{cases}
1 & j-i = -1,\\
-0.1 & (i,j)=(1,n),\\
0 & \text{otherwise.}
\end{cases}$\\
2 & $A_{ij}= \begin{cases}
-1 & j-i =-1,\\
-1 & j-i \in {0,1,2,3},\\
0 & \text{otherwise.}
\end{cases}$ (Matlab's \texttt{gallery('grcar', n)}).\\
3 & Random with independent entries in $N(0,1)$ (Matlab's \texttt{randn(n)}).\\
4 & Random with independent entries in $U([0,1])$ (Matlab's \texttt{rand(n)}).\\
5 & \parbox{0.6\textwidth}{Complex random with independent entries in $N(0,1)$ (Matlab's \texttt{randn(n) + 1i*randn(n)}).}\\
6 & \parbox{0.6\textwidth}{Complex random with independent entries in $U([0,1])$ (Matlab's \texttt{rand(n) + 1i*rand(n)}).}\\
\bottomrule
\end{tabular}
\end{center}
Types 1--4 were used, with the same dimensions, in the paper by Gillis and Sharma~\cite{GS17}. We added types 5--6 to test the complex version of our algorithm as well as the real one.

For each matrix, we computed the nearest Hurwitz stable matrix ($\Omega = \Omega_H$) with $\F=\R$ for matrices of types 1--4 , and $\F=\C$ for matrices of types 5--6. We compared our results with the algorithms BCD, Grad, FGM from~\cite{GS17}, the algorithm SuccConv from~\cite{ONV13}, and a BFGS algorithm based on GRANSO~\cite{CurMO} (which is an improvement of HANSO~\cite{LewO}). The new algorithm presented here is dubbed Orth. The implementations of all competitors are by their respective authors and are taken from Nicolas Gillis' website \url{https://sites.google.com/site/nicolasgillis/code}. We highlight the fact that we are comparing methods that rely on different formulations of the problem (based on different parametrizations of the feasible set); we are not simply plugging in a different general-purpose optimization algorithm.

\pgfplotscreateplotcyclelist{ourlist}{
red, solid, every mark/.append style={solid}, mark=x\\
red, dashed, every mark/.append style={solid}, mark=o\\
red, dotted, every mark/.append style={solid}, mark=star\\
green, solid, every mark/.append style={solid}, mark=square\\
green, dashed, every mark/.append style={solid}, mark=triangle\\
blue!50!black, solid, every mark/.append style={solid}, mark=diamond\\
}

\newcommand{\plotmatrixtype}[3]{
	\def\matrixsize{#1}
	\def\allottedtime{#2}
	\def\matrixtype{#3}
	\begin{tikzpicture}[scale=0.5]
	\begin{semilogxaxis}[title={$\text{type}=\matrixtype, n = \matrixsize$}, width=0.45\textwidth, height=0.35\textheight, xmin=-1, xlabel={$\text{CPU time} / s$}, ylabel={$\norm{A-X_k}_F$}, cycle list name=ourlist, legend pos=outer north east]
		\foreach \algo in {1,...,6}
		{
			\pgfplotstableread{data/perfprof6-\matrixsize-\allottedtime-\matrixtype-\algo.dat}{\algo};
			\addplot+[very thick,mark indices={1,}] table[x=time, y=objfun] {\algo};
		}
	\ifnum\matrixtype=3
		\ifnum\matrixsize=100
			\legend{BCD,Grad,FGM,BFGS,Orth}
		\else
			\legend{BCD,Grad,FGM,SuccConv,BFGS,Orth}
		\fi
	\fi
	\end{semilogxaxis}
	\end{tikzpicture}
}

\begin{figure}
\centering
\begin{tabular}{lll}
\plotmatrixtype{10}{20}{1} & \plotmatrixtype{10}{20}{2} &\plotmatrixtype{10}{20}{3} \\
\plotmatrixtype{10}{20}{4} & \plotmatrixtype{10}{20}{5} & \plotmatrixtype{10}{20}{6}
\end{tabular}
\caption{Hurwitz stability: distance vs. time for different matrices of size $10$.} \label{fig:10}
\end{figure}

\begin{figure}
\centering
\begin{tabular}{lll}
\plotmatrixtype{20}{100}{1} & \plotmatrixtype{20}{100}{2} &\plotmatrixtype{20}{100}{3} \\
\plotmatrixtype{20}{100}{4} & \plotmatrixtype{20}{100}{5} & \plotmatrixtype{20}{100}{6}
\end{tabular}
\caption{Hurwitz stability: distance vs. time for different matrices of size $20$.}
\end{figure}

\begin{figure}
\centering
\begin{tabular}{lll}
\plotmatrixtype{50}{300}{1} & \plotmatrixtype{50}{300}{2} &\plotmatrixtype{50}{300}{3} \\ 
\plotmatrixtype{50}{300}{4} & \plotmatrixtype{50}{300}{5} & \plotmatrixtype{50}{300}{6}
\end{tabular}
\caption{Hurwitz stability: distance vs. time for different matrices of size $50$.}
\end{figure}

\begin{figure}
\centering
\begin{tabular}{lll}
\plotmatrixtype{100}{600}{1} & \plotmatrixtype{100}{600}{2} &\plotmatrixtype{100}{600}{3} \\
\plotmatrixtype{100}{600}{4} & \plotmatrixtype{100}{600}{5} & \plotmatrixtype{100}{600}{6}
\end{tabular}
\caption{Hurwitz stability: distance vs. time for different matrices of size $100$.} \label{fig:100}
\end{figure}

The convergence plots in Figures~\ref{fig:10}--\ref{fig:100} show that the new algorithm converges in vastly quicker time scales on matrices of small size. In addition, in all but one case the local minimizers produced by the new algorithm are of better quality, i.e., they have a smaller objective function $\norm{A-B}_F$ than the ones of the competitors. This can be seen in the plots by comparing the lower horizontal level reached by each line.

We report in Figure~\ref{fig:perfplot} the results of another experiment taken from~\cite{GS17}, which aims to produce a performance profile. The result confirms that in the vast majority of the cases the local minima found by the new algorithm have a lower objective function.

\pgfplotstableread{data/randbest.dat}{\mytable}

\pgfplotstablegetcolsof{\mytable}
\edef\mylastj{\number\numexpr\pgfplotsretval-1}

\pgfplotstablecreatecol[
  create col/assign/.code={%
    \def\myrowmin{1e10}
    \pgfplotsforeachungrouped \x in {0,...,\mylastj}{%
    \pgfkeys{/pgf/fpu,/pgf/fpu/output format=sci}
    \pgfmathparse{min(\myrowmin,\thisrowno{\x})}%
    \pgfkeys{/pgf/fpu=false}%
    \let\myrowmin=\pgfmathresult%
    }%
    \pgfkeyslet{/pgfplots/table/create col/next content}\myrowmin%
  }
]{minperf}\mytable

\pgfplotsinvokeforeach{0,...,\mylastj}{
  \pgfplotstablecreatecol[
    create col/assign/.code={%
      \pgfmathparse{\thisrowno{#1}/\thisrow{minperf}}%
      \pgfkeyslet{/pgfplots/table/create col/next content}\pgfmathresult%
    }
  ]{perf-#1}\mytable
}
\makeatletter
\pgfmathdeclarefunction{colleqx}{2}{%
    \begingroup%
    \c@pgf@countd=0
    \pgfmathfloattoint{#1}%
    \pgfplotstableforeachcolumnelement{perf-\pgfmathresult}\of\mytable\as\mycompval{%
      \pgfmathfloatparsenumber{\mycompval}%
      \pgfmathfloatgreaterthan{\pgfmathresult}{#2}%
      \ifpgfmathfloatcomparison\relax\else\advance\c@pgf@countd by1\fi%
    }%
    \pgfplotstablegetrowsof{\mytable}%
    \pgfmathparse{\c@pgf@countd/\pgfplotsretval}%
    \pgfmath@smuggleone\pgfmathresult%
    \endgroup%
}

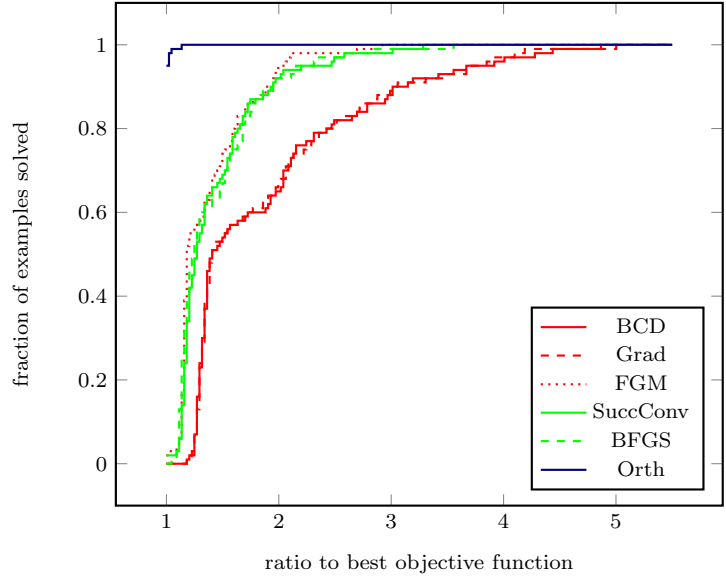
\begin{figure}
\centering
\begin{tikzpicture}
    \begin{axis}[height=0.35\textheight,
                 legend pos=outer north east,no marks,thick,cycle list name=ourlist, legend pos=south east, xlabel={ratio to best objective function}, ylabel={fraction of examples solved}, samples=200]
        \addplot+[const plot,domain=1:5.5] (x,{colleqx(0,x)});
        \addplot+[const plot,domain=1:5.5] (x,{colleqx(1,x)});
        \addplot+[const plot,domain=1:5.5] (x,{colleqx(2,x)});
        \addplot+[const plot,domain=1:5.5] (x,{colleqx(3,x)});
        \addplot+[const plot,domain=1:5.5] (x,{colleqx(4,x)});
        \addplot+[const plot,domain=1:5.5] (x,{colleqx(5,x)});
        \legend{BCD,Grad,FGM,SuccConv,BFGS,Orth}
    \end{axis}
\end{tikzpicture}
\caption{Performance profile of the values of $\norm{A-X}_F$ obtained by the algorithms on 100 random $10\times 10$ matrices (equal split of \texttt{rand} and \texttt{randn}).} \label{fig:perfplot}
\end{figure}

\subsection{Multiple eigenvalues}

Inspecting the minimizers produced by the various methods, we observed that in many cases the new algorithm produces matrices with multiple eigenvalues (especially multiple zero eigenvalues), while other methods settle for worse minimizers with eigenvalues of lower multiplicities (particularly BCD and Grad). An illustrative example is in Figure~\ref{fig:eigenvalues}. We observe that, as illustrated by Example \ref{example:generic}, the instances where the global minimum has multiple eigenvalues are \emph{not} expected to be rare for this problem. This observation provides a qualitative insight to explain why our approach appears to be so highly competitive.

\begin{center}
\begin{figure}[h!]
\begin{tikzpicture}
\begin{axis}[width=0.5\textwidth, cycle list name=ourlist, legend pos = outer north east, title={Eigenvalues of minimizer $B$ for a random $6\times 6$ matrix $A$}]
\addplot+[only marks] table[format=inline]{
   x y
  -1.9552e+00   6.1979e-01
  -1.9552e+00  -6.1979e-01
  -1.1703e-02   2.5450e-01
  -1.1703e-02  -2.5450e-01
   2.1265e-12            0
  -3.6759e-02            0
};
\addplot+[only marks] table[format=inline]{
   x y
  -1.9863e+00   7.7134e-01
  -1.9863e+00  -7.7134e-01
  -7.8281e-03   4.6498e-01
  -7.8281e-03  -4.6498e-01
   2.6966e-13            0
  -2.1897e-02            0
};
\addplot+[only marks] table[format=inline]{
   x y
  -1.9837e+00   7.7292e-01
  -1.9837e+00  -7.7292e-01
  -2.5081e-06   1.4569e-01
  -2.5081e-06  -1.4569e-01
   1.1092e-07            0
  -7.0597e-06            0
};
\addplot+[only marks] table[format=inline]{
   x y
  -1.9837e+00   7.7292e-01
  -1.9837e+00  -7.7292e-01
  -2.5081e-06   1.4569e-01
  -2.5081e-06  -1.4569e-01
   1.1092e-07            0
  -7.0597e-06            0
};
\addplot+[only marks] table[format=inline]{
   x y
  -1.9837e+00   7.7292e-01
  -1.9837e+00  -7.7292e-01
  -2.5081e-06   1.4569e-01
  -2.5081e-06  -1.4569e-01
   1.1092e-07            0
  -7.0597e-06            0
};
\addplot+[only marks] table[format=inline]{
   x y
  -1.9243e+00            0
  -1.9243e+00            0
   6.4514e-05   6.4530e-05
   6.4514e-05  -6.4530e-05
  -6.4514e-05   6.4498e-05
  -6.4514e-05  -6.4498e-05
};
\legend{BCD,Grad,FGM,SuccConv,BFGS,Orth}
\end{axis}
\end{tikzpicture}
\caption{Location of the eigenvalues \texttt{eig(B)} for the local minimizers produced by the methods on a sample random $6\times 6$ matrix $A$. The new method Orth produces a local minimizer with a zero eigenvalue of multiplicity 4 and a negative real eigenvalue of multiplicity 2. Methods FGM, SuccConv, and BFGS produce a different local minimizer (with worse objective function) with a double zero eigenvalue. BCD and Grad produce two (even worse) local minimizers with only one zero eigenvalue and the others having magnitude $\gtrapprox 10^{-2}$.} \label{fig:eigenvalues}
\end{figure}
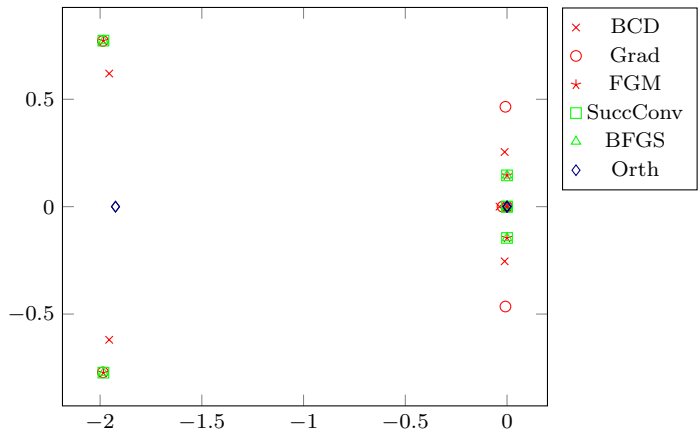
\end{center}

\subsection{Comparison with the method by Guglielmi and Lubich}

Due to the lack of available code, we could not compare directly our method with the algorithm in~\cite{GugL}, which is arguably one of the best competitors. We report some remarks about it based on published figures.

\begin{itemize}
	\item On the matrix \texttt{gallery('grcar', 5)}, the new method computes in 0.1 seconds a different minimizer than the one reported in~\cite{GugL}; both have very close objective values: $\norm{A-B}_F=2.3097$ reported in~\cite{GugL} vs. $2.309628$ found by the new method. We suspect that this minimal difference could simply be due to different stopping criteria.
	\item On \texttt{gallery('grcar', 10)}, the new method computes in 0.3 seconds a minimizer with $\norm{A-B}_F = 3.2834$: this time, undoubtedly a better value than those reported in~\cite{GugL}.
	\item On \texttt{gallery('grcar', 30)}, the new method computes in 3.5 seconds a minimizer with $\norm{A-B}_F = 5.66$, which again improves on the outcomes reported in~\cite{GugL}. The computed minimizer has all its eigenvalues on the imaginary axis: a complex conjugate pair with multiplicity 14 each, and a complex conjugate pair with multiplicity 1 each. In contrast, in~\cite{Gug17slides}, Guglielmi reports a minimizer with $\norm{A-B}_F = 6.57$ and seemingly all distinct eigenvalues all on the imaginary axis, found in 143 seconds (on a different machine, so times cannot be compared directly: still, their very different orders of magnitude suggest that the new algorithm is competitive also in terms of speed).
	\item Guglielmi and Lubich \cite{GugL} report experiments with matrices of size $800$, $1000$, and $1090$. Currently, our method cannot handle those sizes, since the optimization method used stagnates: after an initial improvement it fails to reduce the objective function further, oscillating between various values with nonzero gradient and showing convergence issues in the tCG algorithm used as an inner solver used in the \texttt{trustregions} algorithm. It is an interesting task for future research to study how to improve our algorithm so that it can deal with larger matrices; see Section \ref{sec:conclusions}.
	\item Another advantage of the algorithm in~\cite{GugL} is the ability to deal with many different additional constraints in the form of matrix structure (e.g. sparsity, Toeplitz structure). The new algorithm cannot handle those, and it seems challenging to include them since they do not interact well with the conjugation $Q^\top AQ$ that is a crucial step in our procedure.
\end{itemize}

\subsection{Experiments on Schur stability}

We ran analogous experiments for the case of Schur stability, using (in the case $\F=\R$) Lemma~\ref{lem:schurmin} to solve the $2\times 2$ case. Our terms of comparison were the algorithm FGM from~\cite{GilKS19} and SuccConv from~\cite{ONV13}, again with the code from \url{https://sites.google.com/site/nicolasgillis/code}. The algorithms were run with various choices of initial values as suggested in~\cite{GilKS19}. We refer the reader to that paper for a more detailed description of the other competitors.

Matrices of type 2--6 are the same that were used in Section~\ref{hurwitz-experiments}; type 1 would not make much sense here, since those matrices are already Schur stable for each $n$, so we replaced it with the following, which is a case considered also in~\cite{GilKS19}.
\begin{center}
\begin{tabular}{rl}
\toprule
Type & Description\\
\midrule
1 & all entries equal to 2 (Matlab's \texttt{2*ones(n)}).\\
\bottomrule
\end{tabular}
\end{center}
The results in Figures~\ref{fig:10s}--\ref{fig:100s} confirm the superiority of the new method in this case, too, both in terms of computational time and quality of the minimizers.
\renewcommand{\plotmatrixtype}[4]{
    \def\matrixsize{#1}
    \def\allottedtime{#2}
    \def\matrixtype{#3}
    \begin{tikzpicture}[scale=0.5]
        \begin{axis}[xmode=log, title={$\text{type}=\matrixtype, n = \matrixsize$}, width=0.45\textwidth, height=0.35\textheight, xlabel={$\text{CPU time} / s$}, ylabel={$\norm{A-X_k}_F$}, cycle list name=ourlist, legend pos=outer north east, #4]
            \foreach \algo in {1,...,6}
            {
                \pgfplotstableread{data/perfprofdiscrete-\matrixsize-\allottedtime-\matrixtype-\algo.dat}{\mytable};
                \addplot+[very thick, mark indices={1,}] table[x=time, y=objfun] {\mytable};
            }
        \ifnum\matrixtype=3
            \legend{Stand-FGM,LMI-FGM,mRand-FGM,Stand-SuccConv,LMI-SuccConv,Orth}
        \fi
        \end{axis}
    \end{tikzpicture}
}

\begin{figure}
\centering
\begin{tabular}{lll}
\plotmatrixtype{10}{100}{1}{} & \plotmatrixtype{10}{100}{2}{} &\plotmatrixtype{10}{100}{3}{} \\
\plotmatrixtype{10}{100}{4}{} & \plotmatrixtype{10}{100}{5}{} & \plotmatrixtype{10}{100}{6}{}
\end{tabular}
\caption{Schur stability: distance vs. time for different matrices of size $10$.} \label{fig:10s}
\end{figure}

\begin{figure}
\centering
\begin{tabular}{lll}
\plotmatrixtype{20}{300}{1}{} & \plotmatrixtype{20}{300}{2}{} &\plotmatrixtype{20}{300}{3}{} \\
\plotmatrixtype{20}{300}{4}{} & \plotmatrixtype{20}{300}{5}{} & \plotmatrixtype{20}{300}{6}{}
\end{tabular}
\caption{Schur stability: distance vs. time for different matrices of size $20$.}
\end{figure}

\begin{figure}
\centering
\begin{tabular}{lll}
\plotmatrixtype{50}{600}{1}{} & \plotmatrixtype{50}{600}{2}{} &\plotmatrixtype{50}{600}{3}{} \\ 
\plotmatrixtype{50}{600}{4}{} & \plotmatrixtype{50}{600}{5}{} & \plotmatrixtype{50}{600}{6}{}
\end{tabular}
\caption{Schur stability: distance vs. time for different matrices of size $50$.}
\end{figure}

\begin{figure}
\centering
\begin{tabular}{lll}
\plotmatrixtype{100}{1200}{1}{} & \plotmatrixtype{100}{1200}{2}{ymode=log} &\plotmatrixtype{100}{1200}{3}{} \\
\plotmatrixtype{100}{1200}{4}{} & \plotmatrixtype{100}{1200}{5}{} & \plotmatrixtype{100}{1200}{6}{}
\end{tabular}
\caption{Schur stability: distance vs. time for different matrices of size $100$.} \label{fig:100s}
\end{figure}

In addition, we consider some of the special cases discussed in~\cite{GilKS19} and~\cite{GugP18}. For the matrix
\[
A = \begin{bmatrix}
    0.6 & 0.4 & 0.1\\
    0.5 & 0.5 & 0.3\\
    0.1 & 0.1 & 0.7
\end{bmatrix},
\]
we obtain the same minimizer as~\cite{GilKS19} and~\cite{GugP18} (this is proved to be a global minimizer in~\cite{GugP18}). For the matrix $A = \begin{bsmallmatrix}
    3 & 3\\ 3 & 3
\end{bsmallmatrix}$ we recover one of the two global minimizers given by $\begin{bsmallmatrix}
    1 & 3\\
    0 & 1
\end{bsmallmatrix}$ and its transpose (this is not surprising, because our method based on Lemma~\ref{lem:schurmin} is guaranteed to compute a global minimizer for the real $2\times 2$ case). For the matrix \texttt{2*ones(3,3)}, our method computes the solution
\[
B_1 \approx \begin{bmatrix}
    1.0000 &   0.9550 &   1.5848\\
    1.0450 &   1.0000 &   2.6297\\
    0.4152 &  -0.6297 &   1.0000\\
\end{bmatrix},
\]
with a triple eigenvalue $1$ and $\norm{A-B_1}_F^2$ which is almost exactly $15$ (up to an error of order $10^{-15}$). This matrix is orthogonally similar to
\[
B_2 = \begin{bmatrix}
    1 & 2 & 2\\
    0 & 1 & 2\\
    0 & 0 & 1
\end{bmatrix},
\]
which has $\norm{A-B_2}_F^2=15$ and is very likely a global optimum. A suboptimal solution with entries similar to those of $B_1$ and $\norm{A-B_3}^2 \approx 15.02$ is returned in~\cite{GilKS19}. For the $5\times 5$ matrix appearing in~\cite[Section~4.4]{GugP18}, we obtain a local minimizer $B$ with $\norm{A - B}^2_F \approx 0.5595$, beating the best value 0.5709 reported in~\cite{GilKS19}.

\subsection{Nearest matrix with real eigenvalues}

For a final set of experiments, we have implemented the version of the algorithm that computes the nearest matrix with real eigenvalues. As discussed in Section \ref{sec:realsetup}, in this case we do not need to use the quasi-Schur form, and we can simply take $\mathcal{T}(\hat{A}) = \operatorname{triu}(\hat{A})$, the upper triangular part of $\hat{A}=Q^\top A Q$. Then $\mathcal{L}(\hat{A})= \hat{A} - \mathcal{T}(\hat{A})$ is defined accordingly, and we can use the objective function $f(Q) = \norm{\mathcal{L}(Q^\top A Q)}_F^2$ and the gradient computed in Section~\ref{sec:gradient}. We tested the algorithm on the examples discussed in~\cite{F17}. On the matrix
\[
A = \begin{bmatrix}
    1 & 1 & 0\\
    -1 & 0 & 0\\
    0 & 0 & 0
\end{bmatrix},
\]
the candidates suggested in~\cite{F17} are:

\begin{itemize}
	\item the matrix obtained from the eigendecomposition $A = VDV^{-1}$ as $B_1 = V (\Re D) V^{-1}$, which has distance $\norm{A-B_1}_F = 1.5811$;
	\item the matrix
	\[
	B_2 = \begin{bmatrix}
	    0.9 & 0 & 0 \\
	    -2 & -0.1 & 0\\
	    0 & 0 & 0
	\end{bmatrix},
	\]
	which has distance $\norm{A-B_2}_F = 1.4213$ (this matrix actually is not even a local minimizer);
	\item the matrix obtained from the real Schur factorization $A= QTQ^\top$ as $B_3 = Q\operatorname{triu}(T)Q^\top$, which has distance $\norm{A-B_3}_F = 0.5$. 
\end{itemize}
Our method computes in about 0.1 seconds a minimizer
\[
B_4 \approx \begin{bmatrix}
    1.2110 &   0.7722 &  -0.0962\\
   -0.7722 &  -0.2416 &  -0.0962\\
   -0.0962 &   0.0962 &   0.0306\\
\end{bmatrix},
\]
with a triple eigenvalue at $1/3$ and slightly lower distance $\norm{A-B_4}_F = 0.4946$. So in this example the matrix produced via the truncated real Schur form is very close to being optimal.

Another example considered in~\cite{F17} is the matrix
\[
A = \begin{bmatrix}
    0 & 1 & 0 & 0\\
    -1 & 0 & a & 0\\
    0 & 0 & 0 & 1\\
    0 & 0 & -1 & 0
\end{bmatrix};
\]
we tested this problem setting, for concreteness, $a=10$. The user Christian Remling suggested~\cite{F17} the minimizer (having two double eigenvalues at $\pm 1$)
\[
B_1 = \begin{bmatrix}
    0 & 1 & 0 & 0\\
    -1 & 0 & 10 & 0\\
    0 & 0.4 & 0 & 1\\
    0 & 0 & -1 & 0
\end{bmatrix},
\]
which achieves $\norm{A-B_1}_F = 0.4$, beating the minimizer constructed with the Schur form, which has $\norm{A-B_2}_F = \sqrt{2}$. Our method computes in less than 1 second the minimizer
\[
B_3 \approx \begin{bmatrix}
   -0.0000 &   0.9815 &  -0.0000 &   0.0018\\
   -0.9721 &   0.0000 &  10.0045 &   0.0000\\
   -0.0000 &   0.1907 &   0.0000 &   0.9815\\
    0.0945 &  -0.0000 &  -0.9721 &   0.0000
\end{bmatrix},
\]
with a quadruple eigenvalue at $0$ and $\norm{A-B_3} = 0.2181$.

These results confirm that none of the strategies suggested in~\cite{F17} are optimal, and suggest that minimizers with high-multiplicity eigenvalues are to be expected for this problem as well.

\subsection{The complex case and a conjecture}

An interesting observation is that in all the (limited) examples that we tried, for a real $A\in\mathbb{R}^{n\times n}$ we observed that the solution of 
\begin{equation} \label{realversion}
    B = \arg \min_{S(\Omega_H,n,\R)} \norm{A-X}_F^2    
\end{equation}
was also a solution of the corresponding problem where $X$ is allowed to vary over the larger set of complex matrices, i.e.,
\begin{equation} \label{complexversion}
    \arg \min_{S(\Omega_H,n,\C)} \norm{A-X}_F^2.
\end{equation}
It does not seem obvious to prove that this must be the case, since the set $S(\Omega_H,n,\C)$ is non-convex, and in particular there are examples of matrices such that $X\in\mathbb{C}^{n\times n}$ is Hurwitz stable, but $\Re X$ is not. We formulate it as a conjecture.
\begin{conjecture}
For any $A\in\mathbb{R}^{n\times n}$, the problem~\eqref{complexversion} has a real solution $B$.
\end{conjecture}
To support this conjecture, we prove a weaker result.
\begin{theorem}
Let $B \in \mathbb{R}^{n\times n}$ be a solution of~\eqref{realversion}, i.e., a real nearest Hurwitz stable matrix to a given $A\in\mathbb{R}^{n\times n}$. Then, $B$ has a complex Schur decomposition $UTU^*$ where $T = \mathcal{T}(U^*AU)$, and $U$ is a stationary point of~\eqref{minriemann2}.
\end{theorem}
\begin{proof}
In this proof, we need to relate the quantities computed in Sections~\ref{sec:gradient-real} and~\ref{sec:gradient-complex}; we use a subscript $\R$ or $\C$ to tell them apart, so for instance the formula in the theorem becomes $T_{\C} = \mathcal{T}_{\C}(U^*AU)$.

We start by proving that $T_{\C} = \mathcal{T}_{\C}(U^*AU)$ in the $2\times 2$ case. We assume that $A$ is not already Hurwitz stable (otherwise the result is trivial), and divide into two subcases.
\begin{itemize}
    \item $B$ has real eigenvalues. Then, one can take a real modified Schur decomposition $B = QT_{\R} Q^\top$ in which $Q$ is real and $T_{\R}$ is truly upper triangular and not a $2\times 2$ block. By Theorem~\ref{thm:equivalence2}, $T_{\R}=\mathcal{T}_{\R}(Q^\top A Q)$ and $Q$ is a minimizer of~\eqref{minriemannreal}, hence $\grad f_{\R}(Q) = 0$. Since $B=QT_{\R} Q^\top$ is also a complex Schur form, we can take $Q=U$ and hence $T_{\C} = T_{\R} = \mathcal{T}_{\C}(U^*AU)$.
    \item $B$ has two complex conjugate eigenvalues $\alpha \pm i\beta$. Then, looking at Lemma~\ref{lem:2x2hurwitz} and its proof, we see that we must fall in case~1, because in all the other cases $B$ has real eigenvalues. Hence $\alpha = 0$ and $\trace(A) > 0$. In addition, $B =A - \frac12 \trace(A) I_2$. So $A$ and $B$ have the same eigenvectors, and (taking an arbitrary complex Schur decomposition $B = UT_{\C}U^*$) the matrix $U^*AU = \frac12 \trace(A) I+T_{\C}$ is upper triangular. Thus it is easy to see that $T_{\C} = \mathcal{T}_{\C}(U^*AU)$.
\end{itemize}
We now show that $T_{\C} = \mathcal{T}_{\C}(U^*AU)$ holds also for larger matrices $A \in \mathbb{R}^{n\times n}$. Let $B = QT_{\R}Q^\top$ be a modified real Schur decomposition. Note that to obtain a complex Schur decomposition it is sufficient to reduce the $2\times 2$ diagonal blocks to upper triangular form with an unitary block diagonal matrix $D$, thus obtaining a decomposition with $T_{\C} = D^*T_{\R}D$ and $U = QD$. Since $B$ is a real minimizer, $T_{\R} = \mathcal{T}_{\R}(Q^*AQ)$.

From these equalities it follows that $ D T_{\C} D^* =\mathcal{T}_{\R}(DU^*AUD^*)$. We can use this equality to prove that $T_{\C} = \mathcal{T}_{\C}(U^*AU)$.  Let us split $T_{\C}$ into blocks according to the partition $\mathfrak{I}$ as in~\eqref{partition}. The maps $\mathcal{T}_{\R}$ and $\mathcal{T}_{\C}$ leave unchanged the blocks in the strictly upper triangular part, hence for $I < J$ we have 
\[
D_{II}(T_{\C})_{IJ} D_{JJ}^* = (\mathcal{T}_{\R}(DU^*AUD^*))_{IJ} = D_{II}(U^*AU)_{IJ} D_{JJ}^*,
\]
thus $(T_{\C})_{IJ} = (\mathcal{T}_{\C}(U^*AU))_{IJ} = (U^*AU)_{IJ}$. On diagonal blocks, $(T_{\R})_{II}$ is a distance minimizer, hence by the $2\times 2$ version of this result that we have just proved
\[
(T_{\C})_{II} = \mathcal{T}_{\C}(D_{II}^*(Q^*AQ)_{II}D_{II}) = \mathcal{T}_{\C}((U^*AU)_{II}).
\]

It remains to prove that $\grad f_{\C}(U)=0$. This follows by combining Remarks~\ref{rem:gradvanishes-real} and~\ref{rem:gradvanishes-complex}: if $B\in \mathbb{R}^{n\times n}$ is a minimizer on the reals, then $\grad f_{\R}(Q)=0$ by Theorem~\ref{thm:equivalence2} and $B(A-B)^\top = (A-B)^\top B$; both $A$ and $B$ are real matrices, so $B(A-B)^* = (A-B)^* B$ also holds, and $\grad f_{\C}(U)$ vanishes.

\end{proof}

\section{Conclusions}\label{sec:conclusions}

In this paper, we introduced a method to solve the nearest $\Omega$-stable matrix problem, based on a completely  novel approach rather than improving on those known in the literature. The algorithm has remarkably good numerical results on matrices of sufficiently small size (up to $n\approx 100$), both in terms of computational time compared to its competitors, and of quality of the local minima found. 

We attribute a good part of the success of this method to the fact that it computes eigenvalues only indirectly; thus, it is able to avoid the inaccuracies associated with multiple eigenvalues. Indeed, it is well-known that computing eigenvalues with high multiplicity is an ill-conditioned problem: if a matrix has an eigenvalue $\lambda$ of multiplicity $k$, a perturbation of size $\varepsilon$ will, generically, produce a matrix with $k$ simple eigenvalues at distance $\mathcal{O}(\varepsilon^{1/k})$ from $\lambda$ (see, for instance, \cite[Chapter~2]{Wil88-algebraic-eigenvalue}). The method can, in principle, be generalized to similar problems by defining appropriate Schur-like decompositions; for instance, to the problem of finding matrices with at least a given number $k$ of eigenvalues inside a certain set $\Omega$.

In our future research, we plan to study extensions such as the one mentioned above, and to investigate further the behaviour of the method on larger matrices. Hopefully, convergence on larger matrices can be obtained by adjusting parameters in the optimization method, deriving preconditioners that approximate the Hessian of the objective function $f(Q) = \norm{\mathcal{L}(Q^\top AQ)}_F^2$, or including some techniques borrowed from traditional eigensolvers.

\section*{Acknowledgements}

We are grateful to two anonymous reviewers for their insightful comments.

\bibliographystyle{abbrv}
\bibliography{nearest}

\section*{Declarations}

{\bf Funding} VN acknowledges partial support by the Visiting Fellows Programme of the University of Pisa and support by an Academy of Finland grant (Suomen Akatemian p\"a\"at\"os 331240). FP acknowledges partial support by INdAM/GNCS and by a PRA (Progetto di Ricerca d'Ateneo) of the University of Pisa. \\
{\bf Conflicts of interest/Competing interests} Not applicable.\\
{\bf Availability of data and material} Not applicable.\\
{\bf Code availability} The code for the paper is freely available at \url{https://github.com/fph/nearest-omega-stable}.\\


\end{document}